\documentclass[a4paper,11pt]{article}
\usepackage[utf8]{inputenc}
\usepackage{a4wide}
\usepackage{amsmath, amsthm}
\usepackage{amssymb}
\usepackage{graphicx}
\usepackage{ifpdf}
\usepackage{hyperref}
\usepackage{comment}
\usepackage{mathtools}
\usepackage{arydshln}
\usepackage{multicol}
\usepackage{color}
\usepackage{enumerate}

\allowdisplaybreaks

\usepackage{tkz-graph}
\tikzset{
  VertexStyle/.append style = {shape=circle,draw, fill=black, minimum size=9pt, inner sep=-1pt},
  EdgeStyle/.append style = {-, thick},
  LoopStyle/.append style = {-}}
\usetikzlibrary{arrows.meta}
\tikzset{>={Latex[width=2.5mm,length=2.5mm]}}

\newenvironment{speciaalenumerate}{
\begin{enumerate}[(i)]
  \setlength{\itemsep}{1pt}
  \setlength{\parskip}{0pt}
  \setlength{\parsep}{0pt}
}{\end{enumerate}}

\newcommand{\N}{\mathbb{N}}
\newcommand{\Z}{\mathbb{Z}}

\newcommand{\C}{\mathbb{C}}
\newcommand{\R}{\mathbb{R}}

\newcommand{\FF}{\mathcal{F}}
\newcommand{\K}{\mathcal{K}}
\newcommand{\T}{\mathcal{T}}
\newcommand{\x}{\mathbf{x}}
\newcommand{\y}{\mathbf{y}}

\newcommand{\FFF}{\mathfrak{f}}
\newcommand{\EEE}{\mathfrak{e}}
\newcommand{\BBB}{\mathfrak{b}}

\newcommand{\matchsupheight}{_{\vphantom{0}}}
	
\theoremstyle{plain}
\newtheorem{theorem}{Theorem}[section]
\newtheorem{lemma}[theorem]{Lemma}
\newtheorem{proposition}[theorem]{Proposition}

\newtheorem{corollary}[theorem]{Corollary}

\newtheorem{observation}[theorem]{Observation}

\theoremstyle{definition}
\newtheorem{example}[theorem]{Example}
\newtheorem{definition}[theorem]{Definition}

\theoremstyle{remark}
\newtheorem{remark}[theorem]{Remark}

\title{A Tutte polynomial for maps II: the non-orientable case\footnote{A preliminary version of this paper has been presented at EUROCOMB 2017~\cite{goodall2017tutte}.}}
\author{Andrew Goodall\thanks{Charles University, Prague, Czech Republic. Email: \texttt{andrew@iuuk.mff.cuni.cz}. Supported by Project ERCCZ LL1201 Cores and Czech Science Foundation GA \v{C}R 16-19910S.} \and Bart Litjens\thanks{University of Amsterdam, Netherlands. Email: \texttt{bart\_litjens@hotmail.com}. Supported by the European Research Council under the European Union's Seventh Framework Programme (FP7/2007-2013) / ERC grant agreement n$\mbox{}^{\circ}$ 339109.} \and Guus Regts\thanks{University of Amsterdam,  Netherlands. Email: \texttt{guusregts@gmail.com}. Supported by a NWO Veni grant.}\and Llu\'is Vena\thanks{University of Amsterdam, Netherlands. Email: \texttt{lluis.vena@gmail.com}. Supported by the European Research Council under the European Union's Seventh Framework Programme (FP7/2007-2013) / ERC grant agreement n$\mbox{}^{\circ}$ 339109.}}

\begin{document}
\maketitle
\begin{abstract}
We construct a new polynomial invariant of maps (graphs embedded in a compact surface, orientable or non-orientable), which contains as specializations the Krushkal polynomial, the Bollob\'as--Riordan polynomial, the Las Vergnas polynomial, and their extensions to non-orientable surfaces, and hence in particular the Tutte polynomial. Other evaluations include the number of local flows and local tensions taking non-identity values in a given finite group.
\end{abstract}

\section{Introduction}
In~\cite{tutte54} Tutte defined the \emph{dichromate} of a graph~$\Gamma$ 
as a bivariate polynomial graph invariant that includes the chromatic polynomial of~$\Gamma$ 
and the flow polynomial of~$\Gamma$ 
as univariate specializations. The latter two polynomials can be (and usually are) defined by their evaluations at positive integers.  
Let $\Z_n$ denote the additive cyclic group of order~$n$ and suppose we fix an arbitrary orientation of the edges of $\Gamma$.  A nowhere-zero $\Z_n$-flow of $\Gamma$  is an assignment of non-zero elements of $\Z_n$ to the edges of $\Gamma$ such that Kirchhoff's law is satisfied at each vertex. (It then follows that for any edge cutset the sum of the values on edges in one direction is equal to the sum of values on edges in the other direction. It is also evident that the number of nowhere-zero $\Z_n$-flows is an invariant of the graph $\Gamma$, as this number does not depend on the choice of orientation of edges of $\Gamma$.)
For each positive integer~$n$, the flow polynomial of $\Gamma$ evaluated at~$n$ is equal to the number of nowhere-zero $\Z_n$-flows of~$\Gamma$.
The chromatic polynomial evaluated at $n$ counts the number of proper $n$-colourings of~$\Gamma$. A proper colouring of~$\Gamma$ induces a nowhere-zero $\Z_n$-tension of~$\Gamma$, which is to say an assignment of non-zero elements of~$\Z_n$ to edges of~$\Gamma$ such that, for each closed walk, the sum of the values on forward edges equals the sum of the values on backward edges.  Upon fixing the colour of a vertex in each connected component of $\Gamma$ there is a one-to-one correspondence between proper $n$-colourings and nowhere-zero $\mathbb Z_n$-tensions of~$\Gamma$. 

The dichromate was to become better known as the Tutte polynomial and not only contains as evaluations many other important graph invariants, 
but also extends its domain from graphs to matroids, and has revealed fruitful connections between graphs and many other combinatorial structures, such as the Potts model of statistical physics and, more topologically, knots. Another natural way to extend the domain of graphs is to \emph{maps}, that is, graphs embedded in a compact surface (an {\em orientable map} if the surface is orientable, and a {\em non-orientable map} otherwise).  {\em Local flows} and {\em local tensions} of a map are defined similarly to flows and tensions of a graph, and coincide for a plane map with flows and tensions of the underlying planar graph. 
Furthermore, values in a local flow may be taken from a nonabelian group, as the cyclic ordering of edges around vertices of a map determines in which order to multiply elements together when verifying that Kirchhoff's law holds. (We adopt the convention that in nonabelian groups composition is multiplication, while in abelian groups composition is addition and the identity is zero.) Just as the flow polynomial of a graph evaluated at a positive integer $n$ is equal to the number of nowhere-zero $\Z_n$-flows, so for each finite group $G$ we have a map invariant equal to the number of nowhere-identity local $G$-flows. Local tensions are defined dually (just facial walks rather than all closed walks being involved in the definition: the correspondence of tensions with proper vertex colorings is not preserved, but see~\cite{litjens17}). The question is then whether there is a polynomial map invariant which contains as evaluations the number of nowhere-identity local $G$-flows and the number of nowhere-identity local $G$-tensions, in a similar way to how the Tutte polynomial of a graph contains the flow polynomial and the chromatic polynomial as specializations.

Various extensions of the Tutte polynomial to maps have been defined, notably by Las Vergnas~\cite{vergnas80}, Bollob\'as and Riordan~\cite{bollobas01, bollobas02}, and Krushkal~\cite{krushkal11, butler12}, each of which have properties analogous to those of the Tutte polynomial such as having a deletion-contraction recurrence formula or extending from graphs to matroids (the relevant extension from maps being to $\Delta$-matroids~\cite{CMNR16}). 
However, none of these extensions of the Tutte polynomial contain for every finite group $G$ the number of nowhere-identity local $G$-flows and the number of nowhere-identity local $G$-tensions as evaluations. 
Recently such an extension of the Tutte polynomial to orientable maps,
 called the {\em surface Tutte polynomial}, was discovered by three of the authors together with Krajewski~\cite{goodall16}. The surface Tutte polynomial of an orientable map includes the Krushkal polynomial of an orientable map, and hence the Las Vergnas polynomial and Bollob\'as--Riordan polynomial of an orientable map, as specializations.
 
 In the present paper we extend the domain of the surface Tutte polynomial of~\cite{goodall16} to include non-orientable maps and show that this map invariant contains for every finite group $G$ the number of nowhere-identity local $G$-flows and the number of nowhere-identity local $G$-tensions as evaluations, as well as containing further specializations such as the number of quasi-trees of given genus. 
In this way the surface Tutte polynomial of a map is the analogue of the dichromate of a graph as defined by Tutte to include the number of nowhere-zero $\mathbb Z_n$-flows and  the number of nowhere-zero $\mathbb Z_n$-tensions as evaluations. 
Furthermore, the Kruskhal polynomial of a non-orientable map (as defined by Butler~\cite{butler12}) remains a specialization of the surface Tutte polynomial extended to~maps.
While extending the surface Tutte polynomial of~\cite{goodall16} from orientable maps to non-orientable maps suggested itself as a natural step to take, %
and the theorems we prove are generalizations of the theorems in~\cite{goodall16}, in each case the added complications of non-orientability necessitated a substantial development of technique in order to achieve the required lifting of a theorem about orientable maps to a theorem about maps. (A like remark could be made for Bollob\'as and Riordan's extension of their polynomial from orientable maps~\cite{bollobas01} to include non-orientable maps~\cite{bollobas02}.) 

The surface Tutte polynomial of a map is formally more akin to the universal $V$-function of Tutte~\cite{tutte47, tutte54} (see also~\cite[Chapter IX]{tutte01}) than it is to the dichromate in that it has an unbounded number of variables (although for a given map the number of variables is finite): Tutte's universal $V$-function of a graph has variables indexed by the nullities of the connected components of subgraphs; the surface Tutte polynomial of a map has variables indexed by the orientability and genus of the connected components of submaps.
Having infinitely many variables is unavoidable if the number of nowhere-identity local $G$ flows 
is to be included as an evaluation, because of the way the number of such flows depends on the dimensions of the irreducible representations of $G$ (and not just on the size of $G$). (Tutte~\cite{tutte49} showed the number of nowhere-zero flows of a graph taking values in an additive abelian group depends only on the size of the group.) 
While the surface Tutte polynomial is not itself an invariant of the underlying $\Delta$-matroid of a map, it contains specializations that do have this property, including, apart from the Krushkal polynomial, an as yet unstudied four-variable $\Delta$-matroid invariant that we introduce in Section~\ref{sec:Q}. (In this respect the surface Tutte polynomial is similar to the $U$-polynomial of Noble and Welsh~\cite{NW99}, a graph invariant in infinitely many variables, 
which is not itself an invariant of the underlying matroid of the graph, even though it contains many such matroid invariants as specializations, including the Tutte polynomial.)

\subsubsection*{Organization}

As well as the motivation given by Tutte's definition of the dichromate of a graph, we have also in this paper drawn on Tutte's permutation axiomatization of maps~\cite[Chapter X]{tutte01}, which is particularly well-suited to the study of local flows and local tensions of maps, and moreover permits us to rigorously establish some key properties of map operations. While other representations of maps, such as the chord diagram representation used for example by Bollob\'as and Riordan~\cite{bollobas01}, the ribbon graph representation used by Bollob\'as and Riordan ~\cite{bollobas02}, Krushkal~\cite{krushkal11}, and others, or the combinatorial embedding used for instance by Mohar and Thomassen~\cite{mohar01}, have many advantages, we found Tutte's premap representation of maps the most convenient for our purposes.

An outline of the paper follows to help orient the reader. 

In Section~\ref{section:notation} we start by viewing maps more conventionally as graphs embedded in compact surfaces, introduce numerical map parameters such as the genus, and then proceed in the remainder of the section to describe Tutte's permutation axiomatization of maps, define the (surface) dual of a map and deletion and contraction of edges in maps, and derive properties of these operations. (The proof of Lemma~\ref{lem:genus_submaps} in this already lengthy section is deferred to Appendix~\ref{s.proof_s}.) 

The subject of this paper, the surface Tutte polynomial for maps, is introduced in Section~\ref{sec:surface_Tutte}. 
We derive some elementary properties of the surface Tutte polynomial and show that it includes Butler's extension~\cite{butler12} of the Krushkal polynomial~\cite{krushkal11} to maps (non-orientable as well as orientable). 

In Section~\ref{section:flowstensions} we use Tutte's permutation axiomatization of maps to give a 
streamlined definition of local flows and local tensions of maps taking values in a finite group. 
The key result of this paper is Theorem~\ref{theorem:flow_count}, giving an explicit formula for the number of local flows of a map taking non-identity values in a given finite group. The main steps in the proof of this theorem are given in Section~\ref{subsec:proof of eval}, after stating some of its immediate corollaries in Section~\ref{section:corollariesofflowcount}, including those evaluations of the surface Tutte polynomial that give the number of  nowhere-identity local flows and number of nowhere-identity local tensions of a map. 
One of the two key ingredients needed in the proof of Theorem~\ref{theorem:flow_count} is a combinatorial version of the classification theorem for compact surfaces. As a reference for this theorem, and for the language of cell complexes that is needed to utilize it, we found the relatively new book \cite{gaxu13} suitable. The other key ingredient is a result on counting homomorphisms from the fundamental group of a surface to a given finite group. This result can be found in the literature but is not readily accessible to combinatorialists. We have therefore included a proof in Appendix~\ref{sec:proof of flow count} that is new and only uses elementary representation theory. 

As well as the number of nowhere-identity local flows and nowhere-identity local tensions the surface Tutte polynomial contains other significant map invariants as specializations. In Section~\ref{section:otherevaluations} we consider evaluations analogous to those of the Tutte polynomial of a connected graph that enumerate spanning trees, spanning forests and connected spanning subgraphs. In this section we also introduce a different normalization of the surface Tutte polynomial in Proposition~\ref{prop:surface_tutte_renorm} and a four-variable specialization of it in Definition~\ref{def:Q} similar in form to the Krushkal polynomial.

\section{Graphs, maps and operations on maps}\label{section:notation}

Graphs in this paper are finite but may contain loops and multiple edges. Let $\Gamma = (V, E)$ be a graph with vertex set $V$ and edge set $E$. 
For an edge $e \in E$, the graph $\Gamma\backslash e$ obtained from $\Gamma$ by \emph{deletion} of $e$ is the graph $( V, E\setminus\{e\})$. The graph $\Gamma/e$ obtained from $\Gamma$ by \emph{contraction} of $e$ is defined by first deleting $e$ and then identifying the endpoints of $e$.
If $e$ is a loop in $\Gamma$ then $\Gamma/e=\Gamma\backslash e$.

 \subsection{Graphs embedded into compact surfaces}
A \emph{surface} is a two-dimensional topological manifold. 
By the classification theorem for compact surfaces,  a compact surface is either orientable and homeomorphic to a sphere with $g\geq 0$ handles (connected sum of $g$ tori; a sphere if $g=0$) or non-orientable and homeomorphic to a sphere with $g\geq 1$ cross-caps (connected sum of $g$ real projective planes). 
The non-negative integer $g$ is called the (non-)orientable {\em genus} of the surface. 

Perhaps the most usual way to define a map is as a 2-cell embedding of a graph, see e.g.
~\cite{mohar01}: 
\begin{definition}\label{definition:map}
A $\emph{connected map}$ $M$ is a connected graph $\Gamma$ embedded in a connected surface $\Sigma$ (i.e., considered as a subset $\Gamma \subset \Sigma$) such that
\begin{enumerate}
\item vertices are represented as distinct points in the surface,
\item edges are represented as continuous curves in the surface only intersecting at vertices (endpoints),
\item the complement $\Sigma\setminus\Gamma$ of $\Gamma$ inside $\Sigma$ is a disjoint union of connected components, called the $\emph{faces}$ of $M$.
Each face is homeomorphic to an open disc in $\R^2$. 
\end{enumerate}

\noindent A \emph{map} is a disjoint union of connected maps, each embedded in its own surface.
\end{definition}

\begin{definition}\label{definition:genus} Let $M$ be a connected map embedded in a surface $\Sigma$. The \emph{genus} $g(M)$ of $M$ is the genus of $\Sigma$. 
The {\em Euler genus} $s(M)$ of $M$ is the  Euler genus of $\Sigma$, i.e., $$s(M)=\begin{cases} 2g(M) & \mbox{if $M$ is embedded in  orientable $\Sigma$,}\\
  g(M) & \mbox{if $M$ is embedded in  non-orientable $\Sigma$.}\end{cases}$$
The {\em signed genus} $\bar{g}(M)$ of $M$ is the parameter
$$\bar{g}(M)=2s(M)-3g(M)=\begin{cases} g(M) & \mbox{if $M$ is embedded in orientable $\Sigma$,}\\
  -g(M) & \mbox{if $M$ is embedded in non-orientable $\Sigma$.}\end{cases}$$
\end{definition}

For a connected map $M$ given by embedding a graph $\Gamma = (V,E)$ in a surface $\Sigma$, we identify vertices and edges of $\Gamma$ with their representations in $\Sigma$; the set of vertices of $M$ is thereby identified with $V$ and the set of edges of $M$ with $E$. 
A face of $M$ is identified with the multiset of edges of $\Gamma$ that compose its boundary. The collection of faces of $M$ is denoted by $F$.
Two connected maps are said to be \emph{equivalent} if there exists a homeomorphism between the surfaces
in which the two graphs are embedded, which when restricted to the graphs is a graph isomorphism.  The map $M$ is determined up to equivalence by the triple $(V,E,F)$.  

For a connected map $M= (V,E,F)$, set $v(M) := |V|, e(M) := |E|$ and $f(M): = |F|$.

\noindent The {\em Euler characteristic} of $M$ is defined by 
\[
\chi(M) := v(M) - e(M) + f(M). 
\]
\noindent Euler's formula states that 
$$\chi(M)=2-s(M).$$ 

We extend the parameters $v,e,f,g,s$ and $\chi$ additively over disjoint unions of connected maps to maps that are not connected. Defining for a map $M$ the parameter $k(M)$ to be equal to the number of connected maps of which it is composed, Euler's formula is
\begin{equation}\label{eq:Euler}\chi(M)=2k(M)-s(M).\end{equation}

In Definition~\ref{definition:map} we introduced maps as 2-cell embedded graphs; we shall however find the language of combinatorial maps more convenient for our purposes. 
Tutte's permutation axiomatization~\cite{tutte01} of maps is not only well suited for defining local flows and tensions of a map, but also allows us to define contraction and deletion of edges of a map in a way that permits rigorous proofs of properties of these operations. 
While some of the results we derive in the remainder of this section can be found elsewhere (in different forms), see e.g. \cite{damlien03,EMM13,mohar01,tutte01}, we include them for the sake of completeness.

 \subsection{Premaps and Tutte's permutation axiomatization for maps}\label{sec:premaps}
We draw on Chapter X of Tutte's monograph~\cite{tutte01} in defining maps, orientable or non-orientable, in terms of permutations on a finite set. For permutations $\alpha_1,\dots,\alpha_t$ of the same set, we denote by $\langle \alpha_1,\dots,\alpha_t \rangle$ the group of permutations generated by them. The identity permutation is denoted by~$\iota$. 

\begin{definition} \label{def:premap}
	A {\em connected premap} is an ordered triple $(\theta, \sigma, \tau)$ of permutations, each acting on a set $C$ of $4m$ elements (called {\em crosses}), where $m$ is a nonnegative integer, such that
	\begin{enumerate}[(1)]
		\item $\theta^2=\sigma^2=\iota$ and $\theta\sigma=\sigma\theta$,
		\item for any $a\in C$ the elements $a, \theta a, \sigma a, \theta\sigma a$ are distinct,
		\item $\tau\sigma=\sigma \tau^{-1}$,
		\item for each $a\in C$, the orbits of $a$ and $\sigma a$ under the action of $\langle \tau \rangle$ are distinct,
		\item $\langle \theta,\sigma, \tau\rangle$ acts transitively on $C$.
	\end{enumerate}

If $m=0$, the permutations  $\theta,\sigma$ and $\tau$ are empty; the permutation $\tau$ is, for formal reasons, set equal to the product of two empty cycles, written as $(\;)(\;)$. 
\end{definition}

A \emph{premap} is a union of connected premaps, called its connected components. 

	Two premaps are said to be \textit{equivalent} if there exists a bijection between the sets of crosses that maps the ordered triple of permutations of the one premap to the ordered triple of permutations of the other premap.

Let us return to the notion of a connected map $M$ in the sense of Definition~\ref{definition:map} as a graph 2-cell-embedded in a surface $\Sigma$, and give an account of Tutte's~\cite{tutte01} construction of a connected premap from the connected map $M$, which we adapt from Goulden and Jackson~\cite{GJ97}.  Each edge $e$ of the connected map $M$ is a simple curve in $\Sigma$ and has two ends (each of the ends of the curve after puncturing it) and two sides (since the surface is locally orientable, for each edge two sides can be distinguished). 

A {\em side-end position} of $e$ is one of the four possible pairings of a side and an end of $e$. \footnote{More formally, for each vertex $v$ of $M$ we take an open neighborhood $D_v$ of $v$ in $\Sigma$ with the following properties: $D_v$ is homeomorphic to an open disc in $\R^2$, contains no other vertex of $M$ than $v$, and no edge of $M$ is properly contained in $D_v$. For each edge $e$ incident with $v$, let $x_v$ be a point in $(e \cap D_v)\setminus v$. If $e$ is a loop then we take two points $x_v$ and $y_v$ in $(e \cap D_v)\setminus v$ such that $x_v$ and $y_v$ are not in the same connected component of $(e \cap D_v)\setminus v$. The point $x_v$ captures one of the {\em ends} of $e$. For each chosen point $x_v$, let $B_{x_v} \subset D_v$ be an open neighborhood around $x_v$ (homeomorphic to an open disc in $\R^2$) such that $M \cap B_{x_v} = e \cap B_{x_v}$ and such that $e \cap B_{x_v}$ is connected. Then $e \cap B_{x_v}$ divides $B_{x_v}$ into two regions, $B_{x_v}^1$ and $B_{x_v}^2$, which are the two {\em sides} corresponding to the end $x_v$. A {\em side-end} position of $M$ then is a triple of the form $(B_{x_v}^i,x_v,e)$, for some $i$.} Side-end positions correspond to crosses of a connected premap, and the correspondence extends further:  

\begin{theorem}[\cite{tutte01, GJ97}] \label{thm:Tutte_premap_map}
Let $M$ be a connected map in the sense of Definition~\ref{definition:map}, and let $C$ be a set of $4m$ symbols assigned bijectively to the side-end positions of $M$. Let $\theta$ be the permutation that interchanges symbols at the same side but different ends of an edge, for each edge.
Let $\sigma$ be the permutation that interchanges symbols at the same end but different sides of an edge, for each edge.

\begin{itemize}
\item[(1)] {\rm Vertices:} Let $v$ be a vertex of $M$ and $(a_1, a_2, \cdots , a_{2k})$ the list of symbols encountered in a tour of the side-end positions incident with $v$ starting at an arbitrary symbol $a_1$, in the unique (local) direction such that $a_2=\sigma a_1$. Then the permutation $\tau$ in Definition~\ref{def:premap} is the permutation whose disjoint cycles are associated in pairs with each vertex $v$, and have the form $(a_{1} \; a_{3} \; \cdots a_{2k-1})$ and $(a_{2k} \; a_{2k-2} \; \cdots a_2)=(\sigma a_{2k-1} \; \sigma a_{2k-3} \; \cdots \sigma a_1)$. The degree of $v$ is $k$. (If $k=0$ we have a pair of empty cycles associated with the isolated vertex $v$.)

\item[(2)] {\rm Edges:} For each $a\in C$, the elements of $\{a,\theta a,\sigma a,\theta\sigma a\}$ are the symbols assigned to the four side-end positions of the same edge. 
\item[(3)] {\rm Faces:} Let $f$ be a face of $M$ and $(b_1, b_2, \cdots , b_{2j})$ the list of symbols encountered in a tour of the side-end positions incident with $f$ starting at an arbitrary symbol $b_1$, in the unique (local) direction such that $b_2=\theta b_1$. Then the disjoint cycles of $\varphi:=\tau\theta\sigma$ are associated in pairs with each face $f$, and have the form $(b_{1} \; b_{3} \; \cdots b_{2j-1})$ and $(b_{2j} \; b_{2j-2} \; \cdots b_2)=(\theta b_{2j-1} \; \theta b_{2j-3} \; \cdots \theta b_1)$. The degree of $f$ is $j$. (If $j=0$ we have a pair of empty cycles associated with the isolated face $f$.)

\end{itemize}
\end{theorem}

	Axiom~(4) of Definition~\ref{def:premap} says that the crosses $a$ and $\sigma a$ (same end, different sides of an edge) belong to two different cycles of~$\tau$.
Likewise, the crosses $a$ and $\theta a$ (same side, different ends of an edge) belong to two different cycles of $\varphi=\tau\theta\sigma$. 

Tutte~\cite[X.5]{tutte01} gives a reverse construction to Theorem~\ref{thm:Tutte_premap_map}, building a connected map in the sense of Definition~\ref{definition:map} from a connected premap. We need not therefore distinguish between a map and its associated premap, and we shall use the notions and notation of either representation interchangeably.

A premap $(\theta,\sigma,\tau)$ on a set of crosses $C$, and the map $M$ represented by it, will be denoted by $(\theta, \sigma, \tau;C)$. 
An orbit of $\langle\theta,\sigma\rangle$, consisting of four crosses,  is an {\em edge}, and a pair of orbits of $\langle\tau\rangle$ in which crosses $a$ in one of the orbits appear as $\sigma a$ in the other orbit is a {\em vertex} of $M$. An edge and a vertex are incident if some cross belongs to both. An edge is a {\em loop} of $M$ if its crosses are all contained in one vertex, and a {\em non-loop}\footnote{A {\em link} in Tutte's terminology.} of $M$ otherwise. With this definition of incidence, the vertices and edges (loops and non-loops) of $M$ are the vertices and edges (loops and non-loops) of a graph $\Gamma(M)$, which we shall call the {\em underlying graph} of $M$. The underlying graph of a map $M$ in the sense of Definition~\ref{definition:map} is the graph of which $M$ is an embedding into a surface.

Given a premap $(\theta,\sigma,\tau;C)$ and setting $\varphi=\tau\theta\sigma$, the quadruple $(\sigma,\theta,\varphi;C)$ is again a premap, as it satisfies the axioms of Definition~\ref{def:premap}. The roles of vertices and faces played in Theorem~\ref{thm:Tutte_premap_map} are reversed in $(\sigma,\theta,\varphi;C)$ from their role in $(\theta,\sigma,\tau;C)$; likewise for the roles of sides and ends of an edge. We thus arrive at a simple description of surface duality for maps in terms of premaps:

\begin{definition} \label{d.dual_map}
Let $M=(\theta,\sigma,\tau;C)$ be a map, and let  $\varphi=\tau\theta\sigma$.
The {\em dual} of $M$ is the map $M^*=(\sigma,\theta,\varphi;C)$.
\end{definition}

A loop of a map $(\theta,\sigma,\tau;C)$, equal to an edge $\{a,\sigma a,\theta a,\theta\sigma a\}$ whose crosses appear in only two permutation cycles of $\tau$, is \emph{twisted}
if $a$ and $\theta a$ appear in the same cycle of $\tau$, while it is {\em non-twisted} if $a$ and $\theta \sigma a$ appear in the same cycle of $\tau$.

\begin{example}\label{example:standardbouquets} Up to equivalence there are two connected maps with one vertex and one edge:\\
\begin{speciaalenumerate}\vspace{-4mm}
\item Let $C=\{a,\theta a,\sigma a, \theta\sigma a\}$. Then the premap $(\theta,\sigma, \tau;C)$ in which $$\tau=( a \;\; \theta\sigma a)\:( \theta a \;\; \sigma a),$$
and $$\varphi=\tau\theta\sigma =(a)\:(\theta\sigma a)\:( \theta a )\:( \sigma a),$$
represents a loop on a single vertex  in the plane.
The dual map, switching $\theta$ and $\sigma$, and $\tau$ and $\varphi$, has two vertices of degree one and a face of degree two. 
\item The premap $(\theta,\sigma, \tau;C)$ in which 
 $$\tau=(a \;\;\theta a)\:(\sigma a \;\; \theta\sigma a),$$
and  $$\varphi=\tau\theta\sigma =(a \;\; \sigma a)\:(\theta a \;\; \theta\sigma a),$$
represents a twisted loop on a single vertex (a loop embedded in the projective plane).
The dual map is the same map (that is, it is self-dual). 
\end{speciaalenumerate}
\end{example}

A connected map in the sense of Definition~\ref{definition:map} is {\em orientable} if the surface in which it is embedded is orientable, and {\em non-orientable} if this surface is non-orientable. These terms are extended to maps generally by stipulating that a map is orientable if all of its connected components are orientable, and non-orientable if at least one of its connected components is non-orientable.
The distinction between orientable and non-orientable connected maps translates as follows when maps are represented by premaps: 
\begin{theorem}[Theorem $X.12$ in \cite{tutte01}]\label{theorem:orient}
Let $M = (\theta, \sigma, \tau;C)$ be a connected premap. Then the action of $\langle \theta\sigma, \tau \rangle$ on $C$ is either transitive (that is, there is precisely one orbit), in which case $M$ is non-orientable, or there are precisely two orbits, in which case $M$ is orientable.
\end{theorem}

Theorem~\ref{theorem:orient}, along with the construction in Theorem~\ref{thm:Tutte_premap_map} of a connected premap from a connected map, may be seen to be a reformulation of the fact that orientable surfaces in an orientable ambient space are two-sided, while non-orientable surfaces in an orientable ambient space are one-sided.

\begin{remark}
Let $M = (\theta, \sigma, \tau; C)$ be a connected map. As $\langle \theta\sigma, \tau \rangle = \langle \sigma\theta, \varphi \rangle$, where $\varphi = \tau\theta\sigma$, Theorem \ref{theorem:orient} and the equalities $v(M) = f(M^*), e(M) = e(M^*)$ and $f(M) = v(M^*)$ imply that $M^*$ is embedded in the same surface as~$M$.
\end{remark}

\begin{remark}
Suppose $M = (\theta, \sigma,\tau;C)$ is a connected orientable map. 
Let $D$ be one of the two orbits of $C$ under the action of $\langle \theta\sigma, \tau \rangle$. 
Then the ordered pair of permutations $(\theta\sigma, \tau)$, in which the domains of $\theta\sigma$ and $\tau$ are both restricted to $D$, corresponds with the connected map $(\alpha,\tau)$ on the set of darts $D$ as defined in~\cite{goodall16}. Conversely, every connected orientable map on a set of darts as given by Definition~3.2 in~\cite{goodall16} yields a connected orientable map on a set of crosses~\cite[Theorem $X.13$]{tutte01}, informally speaking by viewing the connected orientable map from either side of the surface in which it is embedded.
\end{remark}

\subsection{Deletion and contraction of edges in maps}\label{section:operations}

Deletion and contraction are (surface) dual operations for plane graphs, and this duality extends to matroids more generally (and in this way the Tutte polynomial of a graph lifts to a matroid invariant). However, for non-plane maps this type of duality fails: surface duality and matroid duality no longer coincide.  
Deletion of an edge in the dual map $M^*$ when interpreted in $M$ yields a definition of edge contraction in maps which differs from contraction of the corresponding edge in the underlying graph.  

The purpose of this section is to give a clear, workable definition of deletion and contraction of edges in maps and to establish properties of these operations that will be used in the sequel. 

\begin{definition}[Map edge deletion\footnote{Definition~\ref{d.deletion} coincides with the definition of deletion for generalized maps~\cite[Section~3, 1-Removal]{damlien03}.}]\label{d.deletion}
	Let $M=(\theta, \sigma,\tau;C)$  be a map, and let \\$e=\{a,\theta a, \sigma a, \theta\sigma a\}$ be an edge of $M$. Then the map $M\backslash e$ obtained from $M$ by \emph{deleting} $e$ is given by the map $(\theta',\sigma',\tau';C')$ in which $C'=C\setminus\{a, \theta a,\sigma a, \theta \sigma a\}$ and the permutations are defined as follows.
	For any $b\in C'$,
	\begin{itemize}
		\item $\theta'b=\theta b$.
				\item $\sigma'b=\sigma b$.
		\item 
		$\tau'b=\tau^jb$, where $j\geq 1$ is the minimum positive integer for which $\tau^jb\notin \{a,\sigma a,\theta a,\theta\sigma a\}$. 
Equivalently, since a cycle of $\tau$ contains at most two crosses among $\{a,\sigma a,\theta a,\theta\sigma a\}$, 
		\[
		\tau'b=\begin{cases}
		\tau b & \text{ if } \tau b\in C',\\
		\tau^2 b  & \text{ if } \tau b\in \{a,\theta a,\sigma a,\theta\sigma a\} \text{ and } \tau^2b\in C',\\
		\tau^3 b & \text{ if } \tau b,\tau^2 b\in \{a,\theta a,\sigma a,\theta\sigma a\}.\\
		\end{cases}
		\]

		\item Empty cycles of $\tau$ remain as empty cycles of $\tau'$, and $\tau'$ has an extra empty cycle for each cycle of $\tau$ that contains only crosses in $\{a,\theta a,\sigma a, \theta\sigma a\}$. 
	\end{itemize}
	
	%
	%
\end{definition}

As described in Theorem~\ref{thm:Tutte_premap_map}, the orbits of $\tau$ in a map $M=(\theta,\sigma,\tau;C)$ occurring in pairs correspond to vertices of $M$. To a vertex $v$ of degree $d$ corresponds a pair of orbits $(\;Z\;)$ and $(\;Z^{-1}\;)$, where $Z=(c_1 \; c_2 \; \ldots \; c_d)$ is a sequence of $d$ crosses and $Z^{-1}:=(\sigma c_d \; \ldots \; \sigma c_2 \; \sigma c_1)$. 
To an isolated vertex corresponds a pair of empty cycles $(\:)(\:)$.
Conjugation of the cycle $(\, Z\,)$ cyclically shifts the sequence $Z$ without changing the permutation $\tau$. We shall call this operation {\em rotation} about the vertex $v$.

The pairs of orbits of $\tau$ corresponding to the endvertices of a non-loop 
take the form
\begin{equation}\label{eq:non-loop}\left(\begin{array}{cc} a & X \end{array}\right)\:\left(\begin{array}{cc} \sigma a & X^{-1}\end{array}\right), \quad\mbox{ and }\quad \left(\begin{array}{cc} \theta\sigma a & Y \end{array}\right)\: \left(\begin{array}{cc} \theta a & Y^{-1}\end{array}\right).\end{equation}

The pair of orbits for a vertex incident with a non-twisted loop 
take the form 
\begin{equation}\label{eq:non-twisted_loop}\left(\begin{array}{cccc} a & X & \theta\sigma a & Y\end{array}\right)\:\left(\begin{array}{cccc} \sigma a & Y^{-1} & \theta a & X^{-1}\end{array}\right),\end{equation} where $X, Y$ are (possibly empty) sequences of crosses.

Similarly, 
 the pair of orbits for a vertex incident with a twisted loop 
take the form 
\begin{equation}\label{eq:twisted_loop}\left(\begin{array}{cccc} a & X & \theta a & Y\end{array}\right)\:\left(\begin{array}{cccc} \sigma a & Y^{-1} & \theta\sigma a & X^{-1}\end{array}\right).\end{equation}

\begin{observation}\label{obs:deletion}
According to Definition~\ref{d.deletion}, deleting a non-loop $\{a,\theta a,\sigma a,\theta\sigma a\}$ gives a permutation $\tau'$ equal to $\tau$ restricted to $C\setminus\{a,\theta a,\sigma a,\theta\sigma a\}$ except for the two pairs of orbits~\eqref{eq:non-loop} containing crosses in $\{a,\theta a,\sigma a,\theta\sigma a\}$, which are replaced by  
$$\left(\begin{array}{c} X \end{array}\right)\:\left(\begin{array}{c} X^{-1} \end{array}\right), \quad\mbox{ and }\quad \left(\begin{array}{c} Y \end{array}\right)\:\left(\begin{array}{c} Y^{-1} \end{array}\right).$$ 
Likewise, deleting a non-twisted (twisted) loop $\{a,\theta a,\sigma a,\theta\sigma a\}$ preserves $\tau$ except for the pair of orbits~\eqref{eq:non-twisted_loop} (respectively~\eqref{eq:twisted_loop}) containing crosses in $\{a,\theta a,\sigma a,\theta\sigma a\}$, which are replaced by  
$$\left(\begin{array}{cc} X & Y\end{array}\right)\:\left(\begin{array}{cc} Y^{-1} & X^{-1}\end{array}\right).$$ 
\end{observation}

Deletion of an edge $e$ in a map $M$ with underlying graph~$\Gamma$ gives a map $M\backslash e$ whose underlying graph is $\Gamma\backslash e$. In this sense map edge deletion coincides with graph edge deletion. We have $v(M)=v(M\backslash e)=v(\Gamma\backslash e)=v(\Gamma)$, $e(M)-1=e(M\backslash e)=e(\Gamma\backslash e)=e(\Gamma)-1$ and $k(M\backslash e)=k(\Gamma\backslash e)$. However, deletion of $e$ may reduce the genus, in which case $\Gamma\backslash e$ is embedded in a different surface than $\Gamma$.  


\begin{definition}[Map edge contraction]\label{d.contraction}
	Let $M=(\theta,\sigma,\tau;C)$ be a map, and let $e=\{a,\theta a, \sigma a, \theta\sigma a\}$ be an edge of $M$. Then the map $M/ e$ obtained from $M$ by \emph{contracting} $e$ is the map $(\theta'',\sigma'',\tau'';C'')$ in which $C''=C\setminus\{a,\theta a, \sigma a, \theta \sigma a\}$ and the permutations are defined as follows.
	For any $b\in C''$,
	\begin{itemize}
		\item $\theta'' b=\theta b$.
		\item $\sigma'' b=\sigma b$.

		\item 
		
		$\tau''b=
		\begin{cases}
		\tau b & \text{ if } \tau b\in C'',\\
		\tau \theta\sigma  \tau b & \text{ if } \tau b\in \{a,\sigma a, \theta a, \theta \sigma a\} \text{ and } \tau \theta\sigma \tau b\in C'',\\
		\tau(\theta \sigma  \tau)^2 b & \text{ if } \tau b, \tau\theta\sigma \tau b\in\{a,\sigma a, \theta a, \sigma \theta a\}.\\ 
		\end{cases}$
		
		\item Empty cycles of $\tau$ 
remain as empty cycles of $\tau''$, and 
$\tau''$ has an extra empty cycle for each cycle of $\tau$ that contains only crosses in $\{a,\sigma a, \theta a, \theta\sigma a\}$.  
	\end{itemize}
\end{definition}

\begin{observation}\label{obs:contraction} According to Definition~\ref{d.contraction}, contracting a non-loop $\{a,\theta a,\sigma a,\theta\sigma a\}$ gives a permutation $\tau''$ equal to $\tau$ restricted to $C\setminus\{a,\theta a,\sigma a,\theta\sigma a\}$ except for the two pairs of orbits~\eqref{eq:non-loop} containing crosses in $\{a,\theta a,\sigma a,\theta\sigma a\}$, which are replaced by the single pair 
$$\left(\begin{array}{cc} X & Y\end{array}\right)\:\left(\begin{array}{cc} X^{-1} & Y^{-1}\end{array}\right).$$ 
Contracting a non-twisted loop $\{a,\theta a,\sigma a,\theta\sigma a\}$ preserves $\tau$ except for the pair of orbits~\eqref{eq:non-twisted_loop} containing crosses in $\{a,\theta a,\sigma a,\theta\sigma a\}$, which are replaced by the two pairs  
$$\left(\begin{array}{c} X \end{array}\right)\:\left(\begin{array}{c} X^{-1}\end{array}\right)\quad\mbox{ and }\quad\left(\begin{array}{c} Y\end{array}\right)\:\left(\begin{array}{c} Y^{-1}\end{array}\right).$$ 
Finally, contracting a twisted loop $\{a,\theta a,\sigma a,\theta\sigma a\}$ preserves $\tau$ except for the pair of orbits~\eqref{eq:twisted_loop} containing crosses in $\{a,\theta a,\sigma a,\theta\sigma a\}$, which are replaced by the pair
$$\left(\begin{array}{cc} X & Y^{-1}\end{array}\right)\:\left(\begin{array}{cc} X^{-1} & Y\end{array}\right).$$

 \end{observation}

Contraction of a non-loop or twisted loop $e$ in a map $M$ with underlying graph $\Gamma$ gives a map $M/e$ whose underlying graph is $\Gamma/e$ (equal to $\Gamma\backslash e$ when $e$ is a loop). When $e$ is a non-twisted loop, however, the underlying graph of $M/e$ is not isomorphic to $\Gamma/e$, as it has one more vertex. 

Rewriting Definition~\ref{d.contraction} in terms of $\varphi=\tau\theta\sigma$ makes it apparent that contraction of an edge 
 of a map 
 corresponds to deletion of the same edge in the dual map. 

\begin{proposition}
\label{p.del_dual_cont}
	Let $M=(\theta,\sigma,\tau;C)$ be a map and $M^{\ast}=(\sigma,\theta,\tau\theta\sigma;C)$ its dual.
	For an edge $e=\{a, \theta a,\sigma a, \theta\sigma a\}$, we have 
$(M/e)^{\ast}=M^{\ast}\backslash e$ and $(M\backslash e)^{\ast}=M^{\ast} / e$.
\end{proposition}

\begin{proof}
	We show that $(M/e)^{\ast}=M^{\ast}\backslash e$; the other case follows by the identity $(M^*)^*=M$ directly from the first.	
	Using the notation of Definition~\ref{d.contraction}, 
	if $M=(\theta,\sigma,\tau;C)$ then
	$M/e=(\theta'',\sigma'',\tau'';C'')$, with $C''=C\setminus e$.
	Hence, $(M/e)^{\ast}=(\sigma'',\theta'',\tau''\theta''\sigma'';C'')$.
	On the other hand, $M^{\ast}=(\sigma,\theta,\tau\theta\sigma;C)$ by Definition~\ref{d.dual_map}. Therefore 
	$M^{\ast}\backslash e=(\sigma',\theta',(\tau\theta\sigma)';C')$ with $C'=C''$.
	
The permutations $\theta'$ and $\theta''$ are equal as they are both equal to the restriction of $\theta$ to $C\backslash e$, and likewise  $\sigma'=\sigma''$. Thus it remains to show that $(\tau\theta\sigma)'=\tau''\theta''\sigma''$.
	
	For $b\in C\backslash e$, we have $c:=\theta\sigma b =\theta''\sigma'' b\in C\backslash e$. Then we have, by Definition~\ref{d.contraction} applied to $c$,
	\[
	\tau'' c=
	\begin{cases}
	\tau c & \text{ if } \tau c\in C'',\\
	\tau \theta\sigma  \tau c  & \text{ if } \tau c\in \{a,\theta a,\sigma a,  \theta \sigma a\} \text{ and } \tau \theta\sigma \tau c\in C'',\\
	\tau(\theta \sigma  \tau)^2 c & \text{ if } \tau c,\tau\theta\sigma\tau c\in \{a,\sigma a, \theta a, \sigma \theta a\},\\ 
	\end{cases}
	\]
	and, by Definition~\ref{d.deletion} and Definition~\ref{d.dual_map} applied to $b$, 
	\[
	(\tau\theta\sigma)' b=
	\begin{cases}
	(\tau\theta\sigma) b & \text{ if } (\tau\theta\sigma) b\in C'',\\
	(\tau\theta\sigma)^2 b & \text{ if } (\tau\theta\sigma)b\in \{a,\sigma a, \theta a, \theta\sigma a\} \text{ and } (\tau\theta\sigma)^2 b\in C'',\\
	(\tau\theta\sigma)^3 b & \text{ if } (\tau\theta\sigma)b,(\tau\theta\sigma)^2b \in \{a,\sigma a, \theta a, \theta\sigma a\}.\\ 
	\end{cases}
	\]
	Since $c=\theta\sigma b$, we conclude that $(\tau\theta\sigma)'=\tau''\theta''\sigma''$.
	
	Finally, $(\tau\theta\sigma)'$ and $\tau''\theta''\sigma''$ have the same number of empty cycles, since in both cases we add empty cycles for those cycles in $\tau\theta\sigma$ containing crosses exclusively in $\{a,\sigma a, \theta a, \theta\sigma a\}$.
\end{proof}


Just as for ordinary graphs, the order in which edges are contracted and deleted is immaterial.
\begin{lemma}\label{l.order_immaterial}
	Given sets of edges $A,B$ of a map $M$ with $A\cap B=\emptyset$, then $(M/A)\backslash B$ is well defined,  and $(M/A)\backslash B=(M\backslash B)/A$. 
\end{lemma}
\begin{proof}
It suffices to show that that $(M\backslash e)/f=(M/f)\backslash e$ for two distinct edges $e$ and $f$.	
	Let $e=\{a,\theta a,\sigma a, \theta\sigma a\}$ and $f=\{b,\theta b,\sigma b,\theta\sigma b\}$. Then the map after deleting $e$ and contracting $f$ is the map $(\theta''',\sigma''',\tau''';C''')$ in which $C'''=C\setminus \{a,\theta a,\sigma a, \theta\sigma a,b,\theta b,\sigma b, \theta\sigma b\}$ and the permutations are defined as follows. For $c\in C'''$ we have $\theta'''c=\theta c$, $\sigma'''c=\sigma c$. The permutation $\tau'''$ is obtained by the following procedure. For $c \in C'''$,
	\begin{enumerate}
		\item\label{st.1} Set $d:=\tau c$. Move to the next step.
		\item If $d\in\{a,\theta a,\sigma a, \theta\sigma a\}$, then $c:=\tau d$ and move to step \ref{st.1}. Otherwise move to the next step.
		\item If $d\in \{b,\theta b,\sigma b,\theta\sigma b\}$, then $c:=\tau\theta \sigma d$ and move to step \ref{st.1}. Otherwise move to the next step.
		\item Set $\tau'''c:=d$.
	\end{enumerate}
Since $e \cap f = \emptyset$, the conditions of steps 2 and 3 are mutually exclusive. Hence, the order of these steps can be switched without affecting the outcome of the procedure. It follows that the procedure does not depend on whether we delete $e$ first and then contract~$f$ or vice versa. 
\end{proof}

We conclude this section by recording the effect of contracting a non-loop on various map parameters, starting with a result shown by Tutte for orientability.

\begin{lemma}[\cite{tutte01}, Theorem X.26]\label{l.cont_edge_orient}
Let $M$ be a connected map and let $e$ be a non-loop of $M$. 
Then $M$ is orientable if and only if $M/e$ is orientable.
\end{lemma}

	\begin{lemma}\label{l.param_edge_cont_maint}
		If $e$ is a non-loop of a map $M$ then $v(M/e)=v(M)-1$, $e(M/e)=e(M)-1$, $f(M/e)=f(M)$, and $k(M/e)=k(M)$.
		In particular for a connected map $M$ we have $\chi(M/e)=\chi(M)$, $g(M/e)=g(M)$, and $s(M/e)=s(M)$.
		
		Dually, if $e$ is a non-loop in $M^{\ast}$ then $v(M\backslash e)=v(M)$, $e(M\backslash e)=e(M)-1$, $f(M\backslash e)=f(M)-1$, and $k(M\backslash e)=k(M)$. In particular for a connected map $M$ we have $\chi(M\backslash e)=\chi(M)$, $g(M\backslash e)=g(M)$, and $s(M\backslash e)=s(M)$. 
	\end{lemma}

\begin{proof}
	Let $e = \{a,\theta a,\sigma a, \theta\sigma a\}$ be as in the first assertion. It is clear that $e(M/e)=e(M)-1$.
	Since $e$ is a non-loop each cross belongs to a different cycle of $\tau$. By Observation~\ref{obs:contraction} 
contraction of $e$ merges the cycle containing $a$ with that containing $\theta\sigma a$, and the cycle containing $\theta a$ with the one containing $\sigma a$. 
Thus $v(M/e)=v(M)-1$.
	
The number of connected components of $M$ stays the same upon contracting a non-loop~$e$. Indeed, if $M$ has underlying graph $\Gamma$ then $M/e$ has underlying graph $\Gamma/e$, and so $k(M/e)=k(\Gamma/e)=k(\Gamma)=k(M)$.  

To show that $f(M/e)=f(M)$ we make the straightforward observation that $v(M\backslash e)=v(M)$ and, since $M/e=(M^{\ast}\backslash e)^{\ast}$ by Proposition~\ref{p.del_dual_cont}, we have $f(M/e)=v(M^{\ast}\backslash e)=v(M^\ast)=f(M)$. 

We now have that $\chi(M/e)=\chi(M)$. Euler's formula~\eqref{eq:Euler} then yields that $s(M/e)=s(M)$, and by Lemma~\ref{l.cont_edge_orient} it follows that $g(M/e)=g(M)$, finishing the proof of the first statement.
 
	
The dual statement now follows upon applying Proposition~\ref{p.del_dual_cont} and using Observation~\ref{obs:contraction}. 
\end{proof}

Our last technical lemma concerns the Euler genus.

\begin{lemma}\label{lem:genus_submaps} 
For a subset $A$ of edges of a map $M$, 
\begin{equation}\label{equation:sgenus}
s(M\backslash A^c)+s(M/A)\leq s(M),
\end{equation}
where $A^c := E\setminus A$. Furthermore, there is equality in (\ref{equation:sgenus}) if and only if
\[
k(M\backslash A^c) -k(M) - f(M\backslash A^c) +k(M/A)=0 \text{ and }k(M/A)-k(M)-v(M/A)+k(M\backslash A^c)=0.
\]
\end{lemma}
\noindent We refer to Appendix~\ref{s.proof_s} for a proof.
A simple but useful corollary of Lemma~\ref{lem:genus_submaps} is that neither deletion nor contraction of edges increases the Euler genus.

\begin{corollary}\label{l.cor}
	For a subset $A$ of edges of a map $M$,
	$s(M\backslash A^c)\leq s(M)$ and $s(M/A)\leq s(M)$.
\end{corollary}

\section{The surface Tutte polynomial}\label{sec:surface_Tutte}

First we recall how the Tutte polynomial of a graph is defined. 

The number of vertices, edges and connected components of a graph 
$\Gamma$ are denoted by $v(\Gamma), e(\Gamma)$ and $k(\Gamma)$ respectively; $r(\Gamma) := v(\Gamma) - k(\Gamma)$ is the rank of $\Gamma$ and $n(\Gamma) := e(\Gamma) - r(\Gamma)$ its nullity.

The Tutte polynomial $T(\Gamma;x,y)$ of a graph $\Gamma=(V,E)$ is given by the following bivariate subgraph expansion
\begin{equation}\label{eq:tutte}
T(\Gamma;x,y) := \sum_{A \subseteq E}(x-1)^{r(\Gamma)-r(\Gamma\backslash A^{c})}(y-1)^{n(\Gamma\backslash A^{c})},
\end{equation}
where $A^{c} := E\setminus A$ is the complement of $A \subseteq E$.

The \emph{rank} and \emph{nullity} of a map $M$ are, similarly as for graphs, defined by 
\[r(M) := v(M) - k(M)\;\text{ and }\;n(M) := e(M) - r(M).\] 
The \emph{dual rank} and \emph{dual nullity} are defined by 
\[r^*(M) = f(M) - k(M)\; \text{ and }\;n^*(M) = e(M) - r^*(M).\]
Whereas the rank and nullity of a map coincide with the rank and nullity of its underlying graph, the dual rank and dual nullity implicitly involve the Euler genus. For graphs, rank and nullity are related by (matroid) duality, with $r(\Gamma^*)=n(\Gamma)$ for a planar graph $\Gamma$, while for rank and nullity as map parameters we have by Euler's formula~\eqref{eq:Euler} that $r(M)=n^*(M)-s(M)$ and $n(M)=r^*(M)+s(M)$.

We now have all the definitions needed to introduce the map invariant that is the subject of this paper.
\begin{definition}\label{def:surface_Tutte}
Let $\x = (x;\dots, x_{-2},x_{-1}, x_0,x_1,x_2,\dots)$ and $\y = (y; \dots, y_{-2},y_{-1},y_0,y_1,y_2,\dots)$ 
be infinite sequences of indeterminates (variables indexed by $\mathbb Z$ apart from the first). The \emph{surface Tutte polynomial} of a map $M=(V,E,F)$ is the multivariate polynomial 
\begin{equation}\label{eq:surface_Tutte}
\T(M;\x,\y) := \sum_{A \subseteq E}x^{n^*(M/A)}y^{n(M\backslash A^c)} \prod_{\substack{\mathrm{conn. }\text{ }\mathrm{cpts}\\ M_i \text{ }\mathrm{of}\text{ } M/A}}x_{\bar{g}(M_i)} \prod_{\substack{\mathrm{conn. }\text{ }\mathrm{cpts}\\ M_j \text{ }\mathrm{of}\text{ } M\backslash A^c}}y_{\bar{g}(M_j)},
\end{equation}
where $A^c = E\backslash A$ for $A \subseteq E$ and $\bar{g}(M)$ is the signed genus of $M$ (negative when non-orientably embedded, positive when orientably embedded).
\end{definition}

For an orientable map $M$ the surface Tutte polynomial of Definition~\ref{def:surface_Tutte} coincides with the polynomial defined in~\cite[Definition 3.9]{goodall16}. 
Consequently, by~\cite[Proposition 3.14]{goodall16} the surface Tutte polynomial of a plane map $M$ with underlying graph $\Gamma$  reduces to the Tutte polynomial of $\Gamma$, with
\begin{equation}\label{eq:surface_Tutte_planar}
\T(M;\x,\y)=(x_0y_0)^{k(\Gamma)}T(\Gamma;y_0x+1,x_0y+1).
\end{equation}

\begin{proposition}
The surface Tutte polynomial $\T(M;\x,\y)$ of a map $M$ is multiplicative over the connected components of~$M$.
\end{proposition}
\begin{proof}
Deletion and contraction are distributive over disjoint unions,  the parameters $n^{*}$ and $n$ are additive over disjoint unions, and the connected components of a disjoint union of maps are the disjoint union of the connected components of the maps in the disjoint union. 
\end{proof}

Similarly to the Tutte polynomial under matroid duality, the surface Tutte polynomial $\T$ under map duality involves a simple switch of variables.
\begin{proposition}\label{proposition:duality}
For a map $M$ and its dual $M^*$ we have $\T(M;\x,\y) = \T(M^*;\y,\x)$.
\end{proposition}
\begin{proof}
This follows since $n^{\ast}(M)=n(M^{\ast})$ and $(M\backslash A)^{\ast}=M^{\ast} / A$ by Proposition~\ref{p.del_dual_cont}, 
and the range of the summation~\eqref{eq:surface_Tutte} defining $\T(M)$ is all subsets of edges, which is closed under complementation.
\end{proof}

In \cite{butler12} Butler defines an extension of the Krushkal polynomial to graphs embedded in non-orientable surfaces and proves that it contains the Bollob\'as-Riordan polynomial (as defined in~\cite{bollobas02}), the Las~Vergnas polynomial (as defined in~\cite{vergnas80}) and the Tutte polynomial as specializations. Using Butler's definition, the \emph{Krushkal polynomial} of 
a map $M$ is given by 
\[
\K(M;x,y,a^{1/2},b^{1/2}) = \sum_{A \subseteq E}(x-1)^{k(M\backslash A^c)-k(M)}y^{n(M\backslash A^c)}a^{s(M/A)/2}b^{s(M\backslash A^c)/2},
\]
in which $s(M)= 2k(M) - \chi(M)$ is the Euler genus of $M$ (as per Definition~\ref{definition:genus} and Euler's formula~\eqref{eq:Euler}).

\begin{proposition} The surface Tutte polynomial specializes to the Krushkal polynomial. For a map $M$, we have that
\[
\K(M;X,Y,A^{1/2},B^{1/2}) = (X-1)^{-k(M)}\T(M;\x,\y),
\]
in which $\x$ and $\y$ are set equal to the following values: $x=1, x_g = A^{g}, y = Y, y_g = (X-1)B^g$ for $g \geq 0$,  and $x_{-g} = A^{g/2}$ and $y_{-g} = (X-1)B^{g/2}$ for $g \geq 1$.
\end{proposition}
\begin{proof}
Making the given substitution $x\leftarrow 1, \: y\leftarrow Y$ and 
$$x_g\leftarrow A^g,\quad y_g\leftarrow (X-1)B^g, \quad x_{-g}\leftarrow A^{g/2}, \quad y_{-g}\leftarrow (X\!-\!1)B^{g/2}, \quad\quad\mbox{for $g=0,1,2,\dots$}$$  
in the surface Tutte polynomial 
$$\T(M;\x,\y) = \sum_{A \subseteq E}x^{n^*(M/A)}y^{n(M\backslash A^c)} \prod_{\substack{\mathrm{conn. }\text{ }\mathrm{cpts}\\ M_i \text{ }\mathrm{of}\text{ } M/A}}x_{\bar{g}(M_i)} \prod_{\substack{\mathrm{conn. }\text{ }\mathrm{cpts}\\ M_j \text{ }\mathrm{of}\text{ } M\backslash A^c}}y_{\bar{g}(M_j)},$$
gives, using Euler's formula,  the specialization
$$\sum_{A \subseteq E}Y^{n(M\backslash A^c)} \prod_{\substack{\mathrm{conn. }\text{ }\mathrm{cpts}\\ M_i \text{ }\mathrm{of}\text{ } M/A}}A^{s(M_i)/2} \prod_{\substack{\mathrm{conn. }\text{ }\mathrm{cpts}\\ M_j \text{ }\mathrm{of}\text{ } M\backslash A^c}}(X\!-\!1)B^{s(M_j)/2}.$$
Using additivity of the Euler genus over disjoint unions and collecting together the factors in each product this is the same as
$$\sum_{A \subseteq E}Y^{n(M\backslash A^c)}(X-1)^{k(M\backslash A^c)}A^{s(M/A)/2} B^{s(M\backslash A^c)/2}.$$
A comparison with the definition of the Krushkal polynomial above establishes the proposition. 
\end{proof}
As the Krushkal polynomial specializes to the Tutte polynomial, it follows that the surface Tutte polynomial of a map $M$ contains the Tutte polynomial of the underlying graph of $M$ not only when $M$ is plane (equation~\eqref{eq:surface_Tutte_planar}) but for maps in general. For an arbitrary embedding of graph $\Gamma$ in a surface as a map $M$,
$$T(\Gamma;X,Y)=(X-1)^{-k(M)}\mathcal{T}(M;\mathbf x,\mathbf y)$$
in which $\x$ and $\y$ are set equal to the following values: $x=1, y = Y-1$, and $x_g = 1, y_g = X-1$ for $g \in\mathbb Z$.

\section{Enumerating local flows and local tensions}\label{section:flowstensions}
In this section -- returning to the original motivation behind our definition of the surface Tutte polynomial -- we give evaluations of the surface Tutte polynomial that count the number of local flows and local tensions of a map taking non-identity values in any given finite group.

\subsection{Local  flows and local  tensions.}\label{subsection:tensionflow}
We start by defining local flows and tensions for maps represented by triples of permutations on a set of crosses.
\begin{definition}[Local flows and tensions] \label{d.flows_premaps} Let $M=(\theta,\sigma,\tau;C)$ be a map. Let $G$ be a finite group with identity element $1$.

A {\em local $G$-flow} of $M$  is a function $f:C\to G$ with the property that 
\begin{itemize}
\item[(i)] $f(\sigma a)=f(a)^{-1}$ and $f(\theta a)=f(a)$, for each $a\in C$,
\item[(ii)] 
$f(a)f(\tau a)\cdots f(\tau^{-1}a)=1$ for each cycle $(a \; \tau a \; \cdots \; \tau^{-1}a)$ of $\tau$.
\end{itemize}

A {\em local $G$-tension} of $M$  is a function $f:C\to G$ with the property that 
\begin{itemize}
\item[(i)] $f(\theta a)=f(a)^{-1}$ and $f(\sigma a)=f(a)$, for each $a\in C$,
\item[(ii)] 
$f(b)f(\varphi b)\cdots f(\varphi^{-1}b)=1$ for each cycle $( b \; \varphi b \; \cdots \; \varphi^{-1}b)$ of $\varphi:=\tau\theta\sigma.$
\end{itemize}
\end{definition}

By item (1) of Theorem~\ref{thm:Tutte_premap_map}, according to condition (ii) defining a local $G$-flow,  the equation
$$f(a)f(\tau a)\cdots f(\tau^{-1}a)=1$$ 
for the cycle $(a \; \tau a \; \cdots \; \tau^{-1}a)$ of $\tau$ is paired with the equation
$$f(\sigma \tau^{-1}a) \cdots f(\sigma\tau a)f(\sigma a)=1$$
for the cycle $(\sigma \tau^{-1}a \; \cdots \; \sigma\tau a \; \sigma a)$ of $\tau$. By condition (i) defining a local $G$-flow, the latter is equivalent to 
$$f(\tau^{-1}a)^{-1}\cdots f(\tau a)^{-1}f(a)^{-1}=1,$$
that is, the same equation as the former, by taking the inverse of both sides. 

A similar observation holds for the equations defining a local $G$-tension: cycles of $\varphi$ that are paired together (via $\varphi\theta=\theta\varphi^{-1}$) define equations equivalent to each other by taking inverses. 

As $M^*=(\sigma,\theta,\varphi;C)$ is the dual of $M=(\theta,\sigma,\tau;C)$, tension-flow duality is transparent. 
\begin{proposition}\label{proposition:flowdualtension}

A local $G$-tension of a map $M$ is a local $G$-flow of the dual map $M^*$, and conversely. 
\end{proposition}

\begin{remark}
Even if the group $G$ is abelian, local $G$-flows on a map are essentially different from $G$-flows on the underlying graph of the map. For instance, a local $G$-flow on the map $M$ from Example \ref{example:standardbouquets} (ii) corresponds precisely to a solution of the equation $x^2 = 1$ in $G$ (where $G$ is written multiplicatively). However, the assignment of any value in $G$ to the loop in the underlying graph of $M$ yields a $G$-flow of the underlying graph.
\end{remark}

\begin{remark}
Usually (see for instance~\cite{bouchet83}) local $G$-flows of a map $M$ are defined in terms of half-edges of $M$ represented as a signed graph, in which an orientation of half-edges is chosen that is compatible with edge signatures. An assignment of elements of $G$ to the edges of the map $M$ as a signed graph with half-edge orientations is a local flow if at each vertex Kirchhoff's law is satisfied. The number of flows is then shown to not depend on the choice of orientation.
That our definition of a local flow is equivalent to this definition can be shown with the help of Theorem~\ref{thm:Tutte_premap_map}. 
In Remark~\ref{rmk:signed} we briefly discuss local flows for signed graphs taking values in a finite abelian group.
\end{remark}

\subsection{Enumerating local flows and local tensions}
We denote the number of local $G$-flows of $M$ by $q^1_G(M)$ and the number of nowhere-identity local $G$-flows of $M$ by $q_G(M)$.
The number of local $G$-tensions of $M$ is denoted by $p_G^1(M)$ and the number of nowhere-identity local $G$-tensions of $M$ by $p_G(M)$.
Proposition \ref{proposition:flowdualtension} immediately yields the following as a corollary.

\begin{proposition}\label{corollary:flowdualtension}
Let $M$ be a map and $G$ a finite group. Then $q_G(M^*) = p_G(M)$.
\end{proposition}

To determine the number of local flows, we assume some familiarity with the representation theory of finite groups; see for example~\cite{serre77} for background information. Let $G$ be a finite group and $\rho$ a representation of $G$ with character $\chi_{\rho}$. Define 
\begin{equation}\label{fsindicator}
\FF(\rho) := \frac{1}{|G|}\sum_{g \in G}\chi_{\rho}(g^2),
\end{equation}
the \emph{Frobenius-Schur indicator} (see chapter $4$ of \cite{isaacs94} or the section called `The indicator function' in chapter $23$ of \cite{james01}). Frobenius and Schur~\cite{frob06} proved that when restricted to irreducible representations, $\FF$ only takes values in $\{-1,0,1\}$.
Let $\widehat{G}$ denote the set of (equivalence classes of) irreducible representations of~$G$.

\begin{theorem}\label{theorem:flow_count}
	Let $G$ be a finite group with irreducible representations $\rho$ of dimensions $n_{\rho}$ and let $M=(V,E,F)$ be a map. The number  of nowhere-identity local $G$-flows of $M$ is given~by
	
	 \begin{equation}\label{eq:flow count}
	q_G(M) = \sum_{A \subseteq E}(-1)^{|A^c|}|G|^{|A|-|V|}\prod_{\substack{ \mathrm{orient.}\\
			\mathrm{conn.} \;\mathrm{cpts} \\  M_i \; \mathrm{ of }\hspace{1mm} M\backslash A^c}}\sum_{\rho \in \widehat{G}}n_{\rho}^{2-2g(M_i)}
\:\prod_{\substack{\mathrm{non-orient.}\\
		\mathrm{conn.} \;\mathrm{cpts} \\ M_j \; \mathrm{of} \; M\backslash A^c}}\sum_{\rho \in \widehat{G}}\FF(\rho)\matchsupheight^{g(M_j)}n_{\rho}^{2-g(M_j)}.
	\end{equation}
	\end{theorem}

A brief sketch of the proof of Theorem~\ref{theorem:flow_count} goes as follows. First we use Proposition~\ref{corollary:flowdualtension} to reduce the problem of finding the number of local flows for maps to finding the number of local tensions, which will turn out more convenient. Considering the cell complex associated to a (connected) map allows us to use a combinatorial version of the classification theorem for compact surfaces. We are then faced with finding the number of local tensions for `standard bouquets', which are a special type of map on one vertex in which the loops are arranged in a simple way. Counting the number of local tensions for standard bouquets is the content of Theorem~\ref{thm:flows standard bouquets}. Finally, the inclusion-exclusion principle is used to determine the number of nowhere-identity local tensions. 
The exact statements and details of the argument can be found in Section~\ref{subsec:proof of eval}; the more representation-theoretic proof is postponed to Appendix~\ref{sec:proof of flow count} (proof of Theorem~\ref{thm:flows standard bouquets}). 

\subsection{Corollaries to Theorem~\ref{theorem:flow_count}}\label{section:corollariesofflowcount}

Theorem~\ref{theorem:flow_count} implies that the number of nowhere-identity local $G$-flows and, by Propositions~\ref{corollary:flowdualtension} and~\ref{proposition:duality}, the number of nowhere-identity local $G$-tensions can be found as evaluations of the surface Tutte polynomial.
	
	\begin{corollary}\label{t.tutte_flow_evaluation}
		Let $G$ be a finite group with irreducible representations $\rho$ of dimensions $n_{\rho}$ and let $M$ be a map.
	The number  of nowhere-identity local $G$-flows of $M$ is given by
		\begin{equation}\label{equation:numberflow}
		q_G(M) = (-1)^{e(M)-v(M)}\T(M;\x,\y),
		\end{equation}
		in which $\x$ and $\y$ are set equal to the following values:
		$x = 1, y= -|G|$; $x_g = 1$ and $y_g = -|G|^{-1}\sum_{\rho \in \widehat{G}}n_{\rho}^{2-2g}$ for $g \geq 0$; and $x_{-g} = 1$  and $y_{-g} = -|G|^{-1}\sum_{\rho \in \widehat{G}}\FF(\rho)\matchsupheight^gn_{\rho}^{2-g}$ for~$g \geq 1$.\\
		\indent The number of nowhere-identity local $G$-tensions of $M$ 
is given by
		\begin{equation}\label{equation:numbertension}
		p_G(M) =  (-1)^{e(M)-f(M)}\T(M;\y,\x),
		\end{equation}
		with the same $\x$ and $\y$ as above.
		\end{corollary}

The number of nowhere-identity local $G$-flows of a map with $G$  equal to the dihedral group of order eight and with $G$ equal to the quaternion group (also of order eight) coincide, as shown by DeVos~\cite{devos00}. Theorem~\ref{theorem:flow_count}, however, shows that this does not extend to non-orientable maps.
\begin{example}
Let $D_8$ denote the dihedral group of order eight and let $Q_8$ denote the quaternion group. DeVos~\cite{devos00} showed directly that for orientable maps the number of nowhere-identity local $D_8$-flows equals the number of nowhere-identity local $Q_8$-flows. In~\cite{goodall16} and~\cite{litjens17} it was observed that this follows from the formula in equation~\eqref{equation:orientedcase}, as the multiset of dimensions of irreducible representations of $D_8$ and $Q_8$ agree. However, the types of representations according to the Frobenius-Schur indicator differ: $\FF(\rho) = 1$ for every $\rho \in \widehat{D_8}$, while the group $Q_8$ has a unique $2$-dimensional representation $\rho \in \widehat{Q_8}$ for which $\FF(\rho) = -1$. Therefore for a non-orientable map the number of nowhere-identity local $D_8$-flows may differ from the number of nowhere-identity local $Q_8$-flows. An example is given by the standard bouquet (see Definition~\ref{definition:standardbouquet} below or Example~\ref{example:standardbouquets} (ii)) embedded in the projective plane.
\end{example}

The general expression~\eqref{eq:flow count} for the number of nowhere-identity $G$-flows of $M$ in Theorem~\ref{theorem:flow_count} takes a simpler form when restricting $G$ or $M$ to certain classes. We highlight three cases in the following series of corollaries. 

\begin{corollary}
Let $G$ be a finite group for which $\mathcal F(\rho)\neq 0$ for each $\rho\in\widehat{G}$ (for instance, all symmetric groups have this property) and let $M=(V,E,F)$ be a map. Then
\[
q_G(M) = \sum_{A \subseteq E}(-1)^{|A^c|}|G|^{|A|-|V|}\prod_{\substack{\mathrm{conn.} \hspace{1mm}\mathrm{cpts} \\ M_i  \hspace{1mm}\mathrm{ of }\hspace{1mm} M\backslash A^c}}\Big(\sum_{\rho \in \widehat{G}}\FF(\rho)^{\chi(M_i)}n_{\rho}^{\chi(M_i)}\Big).
\]
\end{corollary}
\begin{proof}
In  equation~\eqref{eq:flow count}, if $\mathcal F(\rho)\neq 0$ for each $\rho\in\widehat{G}$ then $\mathcal{F}(\rho)^{2-2g(M_i)}=1$ for an orientable connected component $M_i$ of a submap of $M$, as $\mathcal{F}(\rho)\in\{-1,+1\}$ for each $\rho\in\widehat{G}$.
\end{proof}

If the map $M$ is orientable then for every subset $A$ of the edges of $M$ each connected component of $M\backslash A^c$ is orientable as well and the general expression~\eqref{eq:flow count} for the number of nowhere-identity $G$-flows of $M$ in Theorem~\ref{theorem:flow_count} takes a simpler form, as does its proof (for which see~\cite{goodall16} and~\cite{litjens17}).

\begin{corollary}
	Let $G$ be a finite group and let $M=(V,E,F)$ be an orientable map. The number of nowhere-identity local $G$-flows of $M$ is given by
	\begin{equation}\label{equation:orientedcase}
	q_G(M) = \sum_{A \subseteq E}(-1)^{|A^c|}|G|^{|A|-|V|}\prod_{\substack{\mathrm{conn.} \hspace{1mm}\mathrm{cpts} \\ M_i  \hspace{1mm}\mathrm{ of }\hspace{1mm} M\backslash A^c}}\Big(\sum_{\rho \in \widehat{G}}n_{\rho}^{2-2g(M_i)}\Big).
	\end{equation}
\end{corollary}

Theorem~\ref{theorem:flow_count} also takes a simpler form when $G$ is an abelian group, the irreducible representations of which all have dimension one. 

\begin{corollary}\label{corollary:abeliancase}
Let $G$ be a finite abelian group and let $d$ be the largest integer for which $G$ has a subgroup isomorphic to $\Z_2^d$. Write $|G| = 2^dm$, in which $m \geq 1$. Then, for a map $M=(V,E,F)$, the number of nowhere-zero local $G$-flows of $M$ is given by
\[
q_G(M) = \sum_{A \subseteq E}(-1)^{|A^c|}(2^d)^{|A|-|V|+k(M\backslash A^c)}m^{|A|-|V|+k_{\mathrm{o}}(M\backslash A^c)},
\]
where $k_{\text{o}}(M\backslash A^c)$ denotes the number of orientable connected components of $M\backslash A^c$ and $k(M\backslash A^c)$ the total number of connected components of $M\backslash A^c$.
\end{corollary}
\begin{proof}
Since $G$ is abelian, $n_{\rho} = 1$ for all $\rho \in \widehat{G}$. 
Thus any $\rho\in \widehat{G}$ equals its character $\chi_\rho$.
Let $\rho\in \widehat{G}$ and let $h\in G$. Then, using the fact that $\rho$ is multiplicative and $G$ is abelian, we obtain
\[
\rho(h)^2\FF(\rho)=\frac{1}{|G|}\sum_{g\in G}\rho(h^2)\rho(g^2)=\frac{1}{|G|}\sum_{g\in G}\rho((gh)^2)=\FF(\rho).
\]
This implies that either $\FF(\rho)=0$ or $\rho(h)^2=1$, for all $h \in G$. Hence, if $\FF(\rho) \neq 0$ then $\FF(\rho) = 1$.
For $i\in \{-1,0,1\}$, set $L_i:=\{\rho\in \widehat G\mid \FF(\rho)=i\}$.
Then $L_{-1}$ is empty and $\rho\in L_1$ if and only if $\rho(g)\in \{-1,1\}$ for each $g\in G$.

As $|\widehat{G}| = |G|$, equation~\eqref{eq:flow count} simplifies to
	\[
	q_G(M) = \sum_{A \subseteq E}(-1)^{|A^c|}|G|^{|A|-|V|+k_{\mathrm{o}}(M\backslash  A^c)}|L_{1}|^{k(M\backslash  A^c)-k_{\mathrm{o}}(M\backslash  A^c)}.
	\]
It now remains to determine $|L_1|$. 
	
	By the classification theorem for finite abelian groups we can write 
	\[
	G \cong \prod_{i=1}^{r}\Z_{n_i},
	\]
	for some $r \geq 1$ and $n_i \in \N$ that are prime powers. Let 
	\[
	T := \{i \in [r] \mid n_i \text{ even}\}
	\]
	and let $d = |T|$, which is the largest integer for which $G$ has a subgroup isomorphic to $\Z_2^d$. Then $|G| = 2^dm$, for some $m \geq 1$. We claim that $|L_{1}| = 2^d$. In order to see this, note first that for $\rho \in \widehat{G}$ we can write
	\[
	\rho = \rho_1 \cdots \rho_r,
	\] 
	where $\rho_i \in \widehat{\Z_{n_i}}$ for $i \in [r]$. 
	Now $\rho$ is real-valued if and only if
	\[
	\rho_i  =\begin{cases} \rho_{\text{triv}} \hspace{16.3mm}\text{ for } i \in [r]\setminus T, \\ \rho_{\text{triv}} \text{ or } \rho_{\text{sign}} \hspace{3mm}\text{ for } i \in T,\end{cases}
	\]
where $ \rho_{\text{triv}}$ is the trivial character (taking value just 1) and $ \rho_{\text{sign}}$	is the character of a cyclic group of even order taking values $1$ or $-1$. This is because any irreducible representation of a cyclic group $\Z_{n}$ with $n > 1$ and $2\nmid n$ takes some values in $\C\backslash\R$. This proves the claim and the statement of the corollary follows.
\end{proof}

Corollary~\ref{corollary:abeliancase} may of course be proved directly, without using Theorem~\ref{theorem:flow_count}. 

\begin{remark}\label{rmk:signed}
Corollary~\ref{corollary:abeliancase}, when translated into the language of signed graphs, for which there is a notion of a local $G$-flow when $G$ is abelian, is equivalent to  Theorem~1.3 in~\cite{devos17}. Underlying a map $M$ is a signed graph, namely the underlying graph of $M$ with a positive or negative sign attached to each edge according as it is non-twisted or twisted\footnote{For every vertex, pick one of the two cyclic permutations defining it. Let $E$ be the union of the crosses in these cyclic permutations. Then the edge $e=\{a,\sigma a, \theta a, \theta \sigma a \}$ is twisted if the two crosses appearing in $E$ are either $\{a,\theta a\}$ or $\{\sigma a,\theta \sigma a \}$, and non-twisted if the two crosses are either $\{a,\theta \sigma a\}$ or $\{\sigma a,\theta a \}$.} 
(the signing of edges is the same as that used in e.g.~\cite{mohar01} to define combinatorial embeddings by vertex rotations and edge signs). A cycle is {\em balanced} in a signed graph if the number of negative signs is even, and {\em unbalanced} otherwise; the signed graph as a whole is balanced if every cycle is balanced, otherwise it is unbalanced.  A map is orientable if and only if its underlying signed graph is balanced. 

To define local $G$-flows of a signed graph for an abelian group $G$, the edges are split into two half-edges, each receiving an orientation (aligned when the edge is positive, opposite when the edge is negative) and the same value in $G$ is given to the two half-edges of an edge. For the assignment of values in $G$ to be a local $G$-flow Kirchhoff's law must hold at each vertex (the half-edges directed into the vertex have the same sum of values as the half-edges directed out of the vertex). 
Local $G$-flows of a signed graph correspond to local $G$-flows of a map with this signed graph underlying it. The orientable connected components of a submap correspond to the balanced connected components of the underlying signed subgraph. 
We can now see that Corollary~\ref{corollary:abeliancase} is equivalent to  Theorem~1.3 in~\cite{devos17}.
Furthermore, the expression in Corollary~\ref{corollary:abeliancase} is an evaluation of the following trivariate specialization of the surface Tutte polynomial, which depends only on the underlying signed graph of $M$ and hence yields a signed graph invariant:
\[
S(M; X,Y,Z)=\sum_{A\subseteq E}(X-1)^{k(M\backslash A^c)-k(M)}(Y-1)^{|A|-|V|+k_o(M\backslash A^c)}(Z-1)^{k(M\backslash A^c)-k_o(M\backslash A^c)},
\]
where $k(M\backslash A^c)$ and $k_o(M\backslash A^c)$ denote the number of connected components and orientable connected components of $M\backslash A^c$ respectively, for $A \subseteq E$.  The polynomial $S(M;X,Y,Z)$ is equal to $(X-1)^{-k(M)}T(M;\x,\y)$ in which $\x$ and $\y$ are set equal to the following values: $x=1$, $y=Y-1$, $x_g=1$ for all $g\in \Z$, $y_g=X-1$ if $g\geq 0$ and $y_g=(X-1)(Z-1)/(Y-1)$ if $g\leq -1$.
This polynomial invariant of the underlying signed graph of $M$ is different from the ``signed Tutte polynomial" of  Kauffman~\cite{kauf89} as it is also invariant under switchings of signs at a vertex. 
(This new triviariate Tutte polynomial for signed graphs appears in the slides of a talk~\cite{kriocon13} given by Krieger and O'Connor in 2013, but we have not found any further reference to it.) As a specialization of the surface Tutte polynomial of $M$ to an invariant of the underlying signed graph of $M$ the signed graph invariant  $S(M;X,Y,Z)$ is formed in a similar way to the Tutte polynomial of the underlying graph of~$M$.
\end{remark}

\subsection{Proof of Theorem \ref{theorem:flow_count}}\label{subsec:proof of eval}

To prove Theorem \ref{theorem:flow_count} we require a number of auxiliary definitions and results. 

A {\em bouquet} is a connected map which has just one vertex -- in the representation of $M$ as a premap $(\theta, \sigma, \tau;C)$ the permutation $\tau$ has exactly two cycles. 

\begin{definition}\label{definition:standardbouquet}
A bouquet $M=(\theta,\sigma,\tau;C)$ is a {\em standard bouquet} if the permutation $\tau$ takes the form
\begin{equation}\label{eq:non-orientable sb}
\arraycolsep=3pt
\left(\begin{array}{cccccccccccccc}a_1 & a_2 & \theta\sigma a_1 &  \theta\sigma a_2 &  \cdots &  a_{2g-1} &  a_{2g} &  \theta\sigma a_{2g-1} &  \theta\!\sigma a_{2g} &  a_{2g+1} & \theta\sigma a_{2g+1} & \cdots & a_m& \theta\sigma a_m\end{array}\right)
\end{equation}
 along with its paired cycle (reversed cycle with application of $\sigma$), when it is orientably embedded, and
 \begin{equation}\label{eq:orientable sb}
\arraycolsep=3pt
\left(\begin{array}{cccccccccccc}  
a_1 & \theta a_1 & a_2 & \theta a_2 & \cdots & a_{g} & \theta a_{g} & a_{g+1} & \theta\sigma a_{g+1}  & \cdots & a_m & \theta\sigma a_m\end{array}\right)
\end{equation}
    along with its paired cycle (reversed cycle with application of $\sigma$), when it is non-orientably embedded.
 \end{definition}
For an orientable standard bouquet we have
    $$\varphi=\tau\theta\sigma=
\arraycolsep=3pt
\left(\begin{array}{cccccccccccccc}
a_1 & \theta\sigma a_2 & \theta\sigma a_1 & a_2 & \cdots & a_{2g-1} & \theta\sigma a_{2g} & \theta\sigma a_{2g-1} & a_{2g} & a_{2g+1} & a_{2g+2} & \cdots & a_m\end{array}\right)$$
      $$(\theta\sigma a_{2g+1})\;(\theta\sigma a_{2g+2})\cdots (\theta\sigma a_{m}),$$
      along with the paired cycles (reversed cycles with application of $\theta$).
      Therefore the number of faces is equal to $m-2g+1$. Since there is one vertex and $m$ edges, this gives Euler characteristic $1-m+m-2g+1=2-2g$, and by Euler's formula the (orientable) genus is~$g$.

       For a non-orientable standard bouquet we have
    $$\varphi=\tau\theta\sigma=
\arraycolsep=3pt
\left(\begin{array}{ccccccccccc}  
a_1 & \sigma a_1 & a_2 & \sigma a_2 & \cdots & a_{g} & \sigma a_{g} & a_{g+1} & a_{g+2} & \cdots & a_m\end{array}\right)$$
      $$(\theta\sigma a_{g+1})\;(\theta\sigma a_{g+2})\cdots (\theta\sigma a_{m}),$$
      along with the paired cycles (reversed cycles with application of $\theta$).
      Therefore the number of faces is equal to $m-g+1$, there is one vertex, and $m$ edges. The Euler characteristic is $1-m+m-g+1=2-g$, and by Euler's formula the (non-orientable) genus is~$g$.

\begin{remark}
Let $M$ be a standard bouquet with one face (a {\em canonical map} in Tutte's terminology~\cite[Chapter X]{tutte01}) and with compact surface $\Sigma$. 
If $\Sigma$ is of genus $g$, then its fundamental group $\pi_1(\Sigma)$ has the presentation
\[
\pi_1(\Sigma) \cong \begin{cases} \langle a_1,\dots,a_g,b_1,\dots,b_g \mid a_1b_1a_1^{-1}b_1^{-1}\cdots a_gb_ga_g^{-1}b_g^{-1} = 1 \rangle \hspace{3.5mm}\text{ if } \Sigma \text{ is orientable,}\\ \langle c_1,\dots,c_g \mid c_1^2\cdots c_g^2 = 1 \rangle \hspace{52mm}\text{ if } \Sigma \text{ is not orientable.}\end{cases}
\]
Local $G$-tensions of $M$ therefore correspond one-to-one with homomorphisms from $\pi_1(\Sigma)$ to $G$. The number of surjections from $\pi_1(\Sigma)$ to $G$ can be counted using an algebraic version of the M\"obius inversion formula due to Hall \cite{hall36}. Surjections from $\pi_1(\Sigma)$ to $G$ are in bijection with certain surface coverings of $\Sigma$~\cite{jones95}. 
\end{remark}

\begin{theorem}\label{thm:flows standard bouquets}
Let $M$ be a standard bouquet of dual nullity $n^*$ and genus $g$ and let $G$ be a finite group.
The number $p_G^1(M)$ of local $G$-tensions on $M$ is given by
\[
p_G^1(M) = \begin{cases} |G|^{n^*-1}\sum_{\rho \in \widehat{G}}\FF(\rho)\matchsupheight^gn_{\rho}^{2-g} \hspace{4mm} \text{ if $M$ is non-orientable,}\\
|G|^{n^*-1}\sum_{\rho \in \widehat{G}}n_{\rho}^{2-2g} \hspace{12.4mm} \text{ if $M$ is orientable.}\end{cases}
\]
\end{theorem}
\noindent We provide an elementary proof of Theorem~\ref{thm:flows standard bouquets} in Appendix~\ref{sec:proof of flow count}. 

We can enumerate local $G$-tensions of an arbitrary map $M$ by first applying a sequence of operations reducing $M$ to a standard bouquet of the same dual nullity, genus and orientability as~$M$. 

\begin{theorem}\label{theorem:tensioncount}
Let $M$ be a connected map. Then the number $p_G^{1}(M)$ of local $G$-tensions on $M$ is given by
\begin{equation}
p_G^{1}(M) = \begin{cases} |G|^{n^*(M)-1}\sum_{\rho \in \widehat{G}}\mathcal{F}(\rho)\matchsupheight^{g(M)}n_{\rho}^{2-g(M)} \hspace{2.1mm} \text{if $M$ is non-orientable},\\
|G|^{n^*(M)-1}\sum_{\rho \in \widehat{G}}n_{\rho}^{2-2g(M)} \hspace{16mm} \text{if $M$ is orientable}.\end{cases}
\end{equation}
\end{theorem}

To prove Theorem~\ref{theorem:tensioncount} we use the classification theorem for compact surfaces, which for our purposes is most conveniently formulated in~\cite{gaxu13}. Before stating it we first need some terminology and definitions. For a set $X$, we define $X^{-1} = \{x^{-1} \mid x \in X\}$, the set of formal inverses of elements in $X$. We assume that $X \cap X^{-1} = \emptyset$ and that $(x^{-1})^{-1} = x$, for $x \in X$. The next definition is taken from~\cite{gaxu13}.

\begin{definition}
A \textit{cell complex} $K$ is a triple $K = (\FFF,\EEE,\BBB)$, where $\FFF$ is a finite non-empty collection of \textit{faces}, $\EEE$ is a finite (possibly empty) set of \textit{edges}, and $\BBB$ is the \textit{boundary function} that assigns to a face $A \in \FFF \cup \FFF^{-1}$ a cyclic permutation of edges in $\EEE \cup \EEE^{-1}$ (called the boundary of the face) and that satisfies the following conditions:
\begin{enumerate}
\item If $\BBB(A) = a_1\cdots a_n$, then $\BBB(A^{-1}) = a_n^{-1}\cdots a_1^{-1}$.
\item If $A_1,A_2 \in \FFF$ and $A_1 \neq A_2$, then $\BBB(A_1) \neq \BBB(A_2)$.
\item Every $e \in \EEE \cup \EEE^{-1}$ occurs precisely twice among the elements of all boundaries.
\item The complex $K$ is not the union of two disjoint systems satisfying the above conditions.
\end{enumerate}
\end{definition}

We describe how to construct a cell complex from a map and vice versa. Let $M = (\theta,\sigma,\tau;C)$ be a connected map with $|C| = 4m$, for some $m \geq 1$ (the case of an isolated vertex is trivial). Set $\EEE = \{e_1,...,e_m\}$ and let $T: C \rightarrow \EEE \cup \EEE^{-1}$ be a function satisfying the following two conditions:
\begin{enumerate}
\item If $a, b \in C$ and $a \notin \{b, \theta b, \sigma b, \theta\sigma b\}$, then $T(a) \notin \{T(b),T(b)^{-1}\}$.
\item For all $a \in C$, $T(\theta\sigma a) = T(\theta a) = T(a)^{-1} = T(\sigma a)^{-1}$.
\end{enumerate}
These conditions imply that $T$ is a two-to-one surjection. Let $\FFF = \{A \mid A \in F\}$, the set consisting of formal elements that are the faces of $M$. Then $\FFF^{-1} = F^{-1}$. The boundary function $\BBB$ is defined as follows. If $A$ is a face of $M$ given by 
\[
\left(\begin{array}{ccc} a_1 & \cdots & a_{2m} \end{array}\right)\left(\begin{array}{ccc} \theta a_{2m} & \cdots & \theta a_1 \end{array}\right)\hspace{-.5mm},
\]
then 
\[
\BBB(A) = T(a_1)\cdots T(a_{2m}) \text{ and } \BBB(A^{-1}) = T(\theta a_{2m})\cdots T(\theta a_1).
\]
That $\BBB$ is a boundary function follows from the second condition defining the function~$T$. Hence $(\FFF,\EEE,\BBB)$ is a cell complex.

For the other direction, let $K = (\FFF,\EEE,\BBB)$ be a given cell complex. We may write $\EEE = \{e_1,...,e_m\}$, for some $m \geq 1$. Set $C := \{a_1,...,a_m,a_1',...,a_m',b_1,...,b_m,b_1',...,b_m'\}$. Define two permutations $\theta$ and $\sigma$ on $C$ by
\begin{itemize}
\item $\theta(a_i) = a_i', \hspace{1mm} \theta(b_i) = b_i'$, for all $i$, and $\theta^2 = \iota$.
\item $\sigma(a_i) = b_i, \hspace{1mm} \sigma(a_i') = b_i'$, for all $i$, and $\sigma^2 = \iota$.
\end{itemize}
Next we define the permutation $\varphi$ (corresponding to the faces of the map to be constructed) by describing its collection of pairs of permutation cycles, there being a pair for each face $A \in \FFF$. Since $\varphi = \tau\theta\sigma$, this will uniquely define $\tau$. In the boundaries $\BBB(A)$, where $A$ ranges over $\FFF$, every edge $e \in \EEE$ occurs at least once (if $e^{-1}$ occurs twice, then interchange the role of $e$ and $e^{-1}$). If it occurs twice (overall), replace one of the instances by the formal element $^{\sigma}e$. Let $S: \EEE \cup \EEE^{-1} \rightarrow C$ be the map given by $S(e_i) = a_i$ and $S(e_i^{-1}) = a_i'$, for all $i$. Then for $A \in \FFF$ with $\BBB(A) = e_{i_1}\cdots e_{i_t}$ we define the corresponding cycle in $\varphi$ to be $S(e_{i_1})\cdots S(e_{i_t})$, and we write $\sigma S(e_{i_j})$ whenever $^{\sigma}e_{i_j}$ occurs. It is now easy to verify that $(\theta,\sigma,\tau;C)$ is a connected map.

There are two types of \textit{normal form} of a cell complex. Write $K= (\FFF,\EEE,\BBB)$. Then the two types both have $\FFF = \{A\}$ and are otherwise given by

\begin{speciaalenumerate}
\item $\EEE = \{a_1,...a_g,b_1,...,b_g\}$ and $\BBB(A) = a_1b_1a_1^{-1}b_1^{-1}\cdots a_gb_ga_g^{-1}b_g^{-1}$, with $g \geq 0$.
\item $\EEE = \{a_1,...,a_g\}$ and $\BBB(A) = a_1a_1\cdots a_ga_g$, with $g \geq 1$.
\end{speciaalenumerate}

Via the construction outlined above, the type (i) and type (ii) cell complexes are seen to correspond to the orientable and non-orientable standard bouquet of genus $g$, respectively.

\begin{definition}[Definition $6.3$ in~\cite{gaxu13}]
Let $K$ and $K'$ be cell complexes. Then $K'$ is \textit{an elementary subdivision of} $K$ if $K'$ is obtained from $K$ by one of the following two operations:
\begin{itemize}
\item[(P1)] Two edges $a$ and $a^{-1}$ in $K$ are replaced by $bc$ and $c^{-1}b^{-1}$ in all boundaries, where $b$ and $c$ are distinct new edges not belonging to $K$.
\item[(P2)] A face $A$ in $K$ with boundary $a_1\cdots a_pa_{p+1}\cdots a_n$ is replaced by two faces, $A'$ and $A''$, in $K'$ which have boundaries $a_1\cdots a_pd$ and $d^{-1}a_{p+1}\cdots a_n$ respectively, where $d$ is a new edge not belonging to $K$.  
\end{itemize}
\end{definition}

Operation (P2) for maps corresponds in the dual map to the vertex-splitting operation defined by Tutte in Section X.7 of~\cite{tutte01}. Vertex-splitting is the inverse operation of contracting a non-loop, and leaves the number of local flows on a map invariant. 

The operations (P1) and (P2) both preserve the number of local $G$-tensions up to a factor of $|G|$; more precisely, we have the following result.

\begin{lemma}\label{lemma:gcount}
Let $K$ and $K'$ be two cell complexes and let $M$ and $M'$ denote the maps corresponding to $K$ and $K'$ respectively. If $K'$ is an elementary subdivision of $K$, then 
\[
|G|^{f(M)-e(M)}p^1_G(M) = |G|^{f(M')-e(M')}p^1_G(M').
\]
\end{lemma}
\begin{proof}
On the one hand, if $K'$ is obtained from $K$ by operation (P1) then clearly $f(M') = f(M)$ and $e(M') = e(M)+1$. Furthermore, we have that $p^1_G(M') = |G|p^1_G(M)$, as replacing an edge $a$ in a boundary by $bc$ yields precisely one extra degree of freedom for specifying a local $G$-tension. Hence the lemma follows in this case.
On the other hand, if $K'$ is obtained from $K$ by operation (P2) then $f(M')  = f(M)+1$ and $e(M') = e(M)+1$. Also, $p^1_G(M') = p^1_G(M)$, as every local $G$-tension on $M$ uniquely yields a local $G$-tension on $M'$, and vice versa. This finishes the proof.
\end{proof}

We may now state the classification theorem for compact surfaces.

\begin{theorem}[Theorem 1.1 in~\cite{gaxu13}]\label{theorem:class}
Every cell complex $K$ can be converted to a cell complex in normal form by using operations $(P1)$ and $(P2)$ and their inverses.
\end{theorem}

This allows us to enumerate local $G$-tensions of a map.
\begin{proof}[Proof of Theorem \ref{theorem:tensioncount}]
Let $M$ be a connected map and let $K$ be the cell complex corresponding to $M$. We apply Theorem \ref{theorem:class} to obtain a cell complex $K'$ in normal form, which in turn corresponds to the standard bouquet $M'$ of the same genus and orientability type as map $M$. We then calculate that
\begin{align*}
p^1_G(M) &= |G|^{f(M')-e(M')+e(M)-f(M)}p^1_G(M')\\
&= |G|^{n^*(M)-1}|G|^{-n^*(M')+1}p^1_G(M')\\
&= \begin{cases} |G|^{n^*(M)-1}\sum_{\rho \in \widehat{G}}\mathcal{F}(\rho)\matchsupheight^{g(M')}n_{\rho}^{2-g(M')} \hspace{2mm} \text{ if $M'$ is non-orientable,}\\
|G|^{n^*(M)-1}\sum_{\rho \in \widehat{G}}n_{\rho}^{2-2g(M')} \hspace{16.7mm} \text{ if $M'$ is orientable,}
\end{cases}\\
&= \begin{cases} |G|^{n^*(M)-1}\sum_{\rho \in \widehat{G}}\mathcal{F}(\rho)\matchsupheight^{g(M)}n_{\rho}^{2-g(M)} \hspace{4mm} \text{ if $M$ is non-orientable,}\\
|G|^{n^*(M)-1}\sum_{\rho \in \widehat{G}}n_{\rho}^{2-2g(M)} \hspace{18mm} \text{ if $M$ is orientable,}
\end{cases}
\end{align*}
where we have used Lemma~\ref{lemma:gcount} in the first equality and Theorem~\ref{thm:flows standard bouquets} in the third equality.
\end{proof}

By enumerating nowhere-identity local $G$-tensions we arrive at a proof of Theorem~\ref{theorem:flow_count} (counting nowhere-identity local $G$-flows).
\begin{proof}[Proof of Theorem~\ref{theorem:flow_count}]
Fix a map $M=(V,E,F)$ and a finite group $G$.
Partitioning local $G$-tensions according to the set $A$ of edges on which the tension value equals the identity, we have
\[
p_G^1(M) = \sum_{A \subseteq E}p_G(M\backslash A).
\]
Inclusion-exclusion then gives
\begin{equation}\label{equation:incl}
p_G(M) = \sum_{A \subseteq E}(-1)^{|A^c|}p_G^1(M\backslash A^c),
\end{equation}
where $A^c = E\setminus A$. 

The expression for $p_G(M)$ now follows from equation~\eqref{equation:incl} by 
using Theorem~\ref{theorem:tensioncount} to give an expression for $p_G^1(M\backslash A^c)$ as a product of its values on each of the connected components of $M\backslash A^c$, and using that $n^*(M\backslash A^c)$ is the sum of the dual nullities of the connected components of $M\backslash A^c$. Equation~\eqref{eq:flow count} then follows by duality, using Proposition~\ref{corollary:flowdualtension}.
\end{proof}


\section{Other evaluations of the surface Tutte polynomial}\label{section:otherevaluations}
In this section we give evaluations of the surface Tutte polynomial that are topological analogues of the number of spanning trees and the number of spanning forests, equal to 
evaluations of the ordinary Tutte polynomial. We start with the analogue of spanning trees.
\subsection{Quasi-trees of given genus}
A \emph{quasi-tree} is a connected map which has a single face. In other words, a quasi-tree is the dual map of a bouquet. In particular, for a plane map a quasi-tree is just a spanning tree of the underlying graph.

The following renormalization of the surface Tutte polynomial will be useful for some of the specializations given in this section.
\begin{proposition} \label{prop:surface_tutte_renorm}
Given a  map $M=(V,E,F)$, the specialization $\widetilde{\mathcal{T}}(M;\mathbf x,\mathbf y)$ of the surface Tutte polynomial $\mathcal T(M;\mathbf x,\mathbf y)$ given by replacing $x_g$ by $x^{-2g}x_g$, $x_{-g}$ by $x^{g}x_{-g}$, $y_g$ by $y^{-2g}y_g$ and $y_{-g}$ by $y^{g}y_{-g}$ for $g=0,1,2,\dots$, is equal to
\begin{equation}\label{eq:surface_Tutte_renorm}\widetilde{\mathcal T}(M;\mathbf x,\mathbf y)=\sum_{A\subseteq E}x^{r(M/A)}y^{r^*(M\backslash A^c)}\prod_{\stackrel{\mbox{\rm \tiny conn. cpts}}{\mbox{\rm \tiny $M_i$ of $M/A$}}}x_{\bar{g}(M_i)}\prod_{\stackrel{\mbox{\rm \tiny conn. cpts }}{\mbox{\rm \tiny $M_j$ of $M\backslash A^c$}}}y_{\bar{g}(M_j)},\end{equation}
where $r(M)=v(M)-k(M)$, $r^*(M)=f(M)-k(M)$, $\bar{g}(M)$ is the signed genus of $M$, and $A^c=E\setminus A$ for $A\subseteq E$.  
\end{proposition}
\begin{proof}
The surface Tutte polynomial is by definition given by
$$\T(M;\x,\y) = \sum_{A \subseteq E}x^{n^*(M/A)}y^{n(M\backslash A^c)} \prod_{\substack{\mathrm{conn. }\text{ }\mathrm{cpts}\\ M_i \text{ }\mathrm{of}\text{ } M/A}}x_{\bar{g}(M_i)} \prod_{\substack{\mathrm{conn. }\text{ }\mathrm{cpts}\\ M_j \text{ }\mathrm{of}\text{ } M\backslash A^c}}y_{\bar{g}(M_j)},$$
in which $n^*(M)=e(M)-f(M)+k(M)$, $n(M)=e(M)-v(M)+k(M)$, and $\bar{g}(M)$ is the signed genus of $M$.
 Euler's relation~\eqref{eq:Euler} gives $-s(M)=-2k(M)+v(M)-e(M)+f(M)$, in which 
$$-s(M)=\begin{cases} -2\bar{g}(M) & \mbox{ when }\bar{g}(M)\geq 0,\\
\bar{g}(M) & \mbox{ when }\bar{g}(M)<0.\end{cases}$$ Upon making the substitutions
$$x_g\leftarrow x^{-2g}x_g, \:\:x_{-g}\leftarrow x^{g}x_{-g},\quad y_g\leftarrow y^{-2g}y_g, \:\:y_{-g}\leftarrow y^{g}y_{-g}, \quad\quad\mbox{ for $g=0,1,2,\dots$},$$ thereby scaling the variables in the surface Tutte polynomial as defined above, the product over variables $x_g$ becomes
$$\prod_{\substack{\mathrm{orient. \; conn. }\text{ }\mathrm{cpts}\\ M_i \text{ }\mathrm{of}\text{ } M/A}}x^{-2\bar{g}(M_i)}x_{\bar{g}(M_i)}\prod_{\substack{\mathrm{non-orient. \; conn. }\text{ }\mathrm{cpts}\\ M_i \text{ }\mathrm{of}\text{ } M/A}}x^{\bar{g}(M_i)}x_{\bar{g}(M_i)}=x^{-s(M/A)}\prod_{\substack{\mathrm{conn. }\text{ }\mathrm{cpts}\\ M_i \text{ }\mathrm{of}\text{ } M/A}}x_{\bar{g}(M_i)}, $$
and the product over variables $y_g$ becomes
$$\prod_{\substack{\mathrm{orient.\; conn. }\text{ }\mathrm{cpts}\\ M_j \text{ }\mathrm{of}\text{ } M\backslash A^c}}y^{-2\bar{g}(M_j)}y_{\bar{g}(M_j)}\prod_{\substack{\mathrm{non-orient. \;conn. }\text{ }\mathrm{cpts}\\ M_j \text{ }\mathrm{of}\text{ } M\backslash A^c}}y^{\bar{g}(M_j)}y_{\bar{g}(M_j)}= y^{-s(M\backslash A^c)}\prod_{\substack{\mathrm{conn. }\text{ }\mathrm{cpts}\\ M_j \text{ }\mathrm{of}\text{ } M\backslash A^c}}y_{\bar{g}(M_j)},$$
in which we use the additivity of the Euler genus $s$ over disjoint unions of maps. 
An easy calculation shows that $n^*(M)-s(M)=
r(M)$ and
$n(M)-s(M)=
r^*(M)$, and the statement of the proposition now follows.
\end{proof}

\begin{proposition}\label{prop:genus_xy} Let $M$ be a connected map and write $\bar{g}(M)=\bar{g}$. Let $\bar{h}$ be an integer with $| \bar{h}|\leq |\bar{g}|$. 
Then the evaluation
of $\widetilde{\mathcal T}(M;\mathbf x,\mathbf y)$ (defined in Proposition~\ref{prop:surface_tutte_renorm}) at $x=y=0$, $x_{i}=0$ for $i\neq \bar{g}-\bar{h}$, $x_{\bar{g}-\bar{h}}=1$, $y_j=0$ for $j\neq \bar{h}$, and $y_{\bar{h}}=1$, is equal to the number of quasi-trees of $M$ of signed genus $\bar{h}$ (which is also equal to the number of quasi-trees of $M^*$ of signed genus $\bar{g}-\bar{h}$.)
\end{proposition}
\begin{proof}
If $M$ is orientably embedded (then $\bar{g} \geq 0$), so is each submap of $M$. This case is taken care of by Proposition~4.5 in~\cite{goodall16}.

Suppose that $M$ is non-orientably embedded. Then $s(M)=g(M)$ and $\bar{g}(M)=\bar{g}=-g(M)<0$.
We can then assume $\bar{h}\leq 0$ (as $\bar{g}-\bar{h}$ would otherwise be less than $\bar{g}$, and there are no submaps with signed genus less than $\bar{g}$).

Let $A\subseteq E$ be a subset of the edges giving a nonzero contribution to the sum~\eqref{eq:surface_Tutte_renorm} with the given values assigned to the indeterminates $\mathbf x, \mathbf y$.
Then $r(M/A)+r^*(M\backslash A^c)=0$ (from the fact that $x=0=y$), each component of $M/A$ has signed genus $\bar{g}-\bar{h}$ and each component of $M\backslash A^c$ has signed genus $\bar{h}$ (from the fact that $x_{\bar{g}-\bar{h}}=1=y_{\bar{h}}$ while $x_i=0$ for $i\neq \bar{g}-\bar{h}$ and $y_j=0$ for $j\neq \bar{h}$).
Hence each component of $M/A$ has Euler genus at least $|\bar{g}-\bar{h}|$ and each component of $M\backslash A^c$ has Euler genus at least $|\bar{h}|$. 
By additivity of the (non-negative) Euler genus over connected components, this immediately implies that $s(M/A)\geq |\bar{g}-\bar{h}|$ and $s(M\backslash A^c)\geq |\bar{h}|$.
By Lemma~\ref{lem:genus_submaps},
$$|\bar{g}|=s(M)\geq s(M\backslash A^c)+s(M/A)\geq|\bar{g}-\bar{h}|+|\bar{h}|\geq |\bar{g}|,$$
and so equality holds, that is $s(M/A)= |\bar{g}-\bar{h}|$ and $s(M\backslash A^c)= |\bar{h}|$.

Since rank and dual rank take non-negative values, we have $r(M/A)=0=r^*(M\backslash A^c)$, whence 
\begin{equation}\label{eq:vkfk} v(M/A)=k(M/A)\quad\mbox{ and }\quad f(M\backslash A^c)=k(M\backslash A^c).\end{equation}
As $s(M)=s(M/A)+s(M\backslash A^c)$ we know by Lemma~\ref{lem:genus_submaps},
\begin{equation}\label{eq:kf} k(M/A)+k(M\backslash A^c)= k(M)+f(M\backslash A^c),\end{equation}
and, dually, 
\begin{equation}\label{eq:vk} k(M/A)+k(M\backslash A^c)= k(M)+v(M/A).\end{equation}

From equations~\eqref{eq:vkfk} and~\eqref{eq:vk} we have $k(M\backslash A^c)=k(M)=1$ (the latter equality since by assumption $M$ is connected) and from equations~\eqref{eq:vkfk} and~\eqref{eq:kf} we have $k(M/A)=k(M)=1$. 
Hence $M\backslash A^c$ is a quasi-tree and $\bar{h}=\bar{g}(M\backslash A^c)$, while $M^*\backslash A\cong (M/A)^*$ is a quasi-tree and $\bar{g}-\bar{h}=\bar{g}(M/A)=\bar{g}(M^*\backslash A)$.

Conversely, if $M\backslash A^c$ is a quasi-tree of signed genus $\bar{h}$ (or $M^*\backslash A$ a quasi-tree of signed genus $\bar{g}-\bar{h}$) then $A$ contributes $1$ to the sum~\eqref{eq:surface_Tutte_renorm} with the given values assigned to the indeterminates.
Hence for $|\bar{h}|\leq |\bar{g}|$ the given evaluation is equal to $$\#\{A\subseteq E :  f(M\backslash A^c)=k(M\backslash A^c)=1, \bar{g}(M\backslash A^c)=\bar{h}\},$$
that is, the number of quasi-trees of $M$ of signed genus $\bar{h}$.
\end{proof}

\subsection{Quasi-forests}\label{sec:Q}
We now consider a topological analogue of spanning forests. A {\em quasi-forest} of $M$ is a submap of $M$ each of whose connected components is a quasi-tree.
To describe the evaluations of the surface Tutte polynomial that follow, it is convenient to further specialize the polynomial $\widetilde{\mathcal T}(M;\mathbf x,\mathbf y)$ defined in Proposition~\ref{prop:surface_tutte_renorm} to a quadrivariate polynomial, similar in form to the Krushkal polynomial.

\begin{definition}\label{def:Q}
For $g\geq 0$ we set $x_g=a^{2g}$, $x_{-g}=a^g$, $y_g=b^{2g}$ and $y_{-g}=b^g$ in $\widetilde{\mathcal T}(M;\mathbf x,\mathbf y)$ to give the quadrivariate polynomial
\begin{equation}\label{eq:Q_renorm}\widetilde{\mathcal Q}(M;x,y,a,b)=\sum_{A\subseteq E}x^{r(M/A)}y^{r^*(M\backslash A^c)}a^{s(M/A)}b^{s(M\backslash A^c)}.\end{equation}
\end{definition}

\begin{remark} $\Delta$-matroids are to maps as matroids are to graphs~\cite{B89, CMNR16}. 
The parameters of rank, dual rank, and Euler genus are parameters of the underlying $\Delta$-matroid of $M$ (much as the rank and nullity of a graph are parameters of the underlying graphic matroid). 
The polynomial $\widetilde{\mathcal Q}(M;x,y,a,b)$, thus involving just parameters of the underlying $\Delta$-matroid of $M$,  may be extended from maps to $\Delta$-matroids more generally. See~\cite[Remark 4.11]{goodall16} for an elaboration of this remark (in the context of orientable maps, but the observations there hold for non-orientable maps too). The surface Tutte polynomial of $M$, on the other hand, is not an invariant of the underlying $\Delta$-matroid of $M$, as the disjoint union of maps $M_1$ and $M_2$ is not distinguishable by its $\Delta$-matroid from the map obtained by fusing a vertex from $M_1$ with a vertex from $M_2$ (likewise, Tutte's universal $V$-function is not a matroid invariant, while the Tutte polynomial is a matroid invariant).  
\end{remark}

\noindent By equation~\eqref{eq:surface_Tutte_planar}, if $M=(V,E,F)$ is a plane embedding of $\Gamma=(V,E)$ then
\[
T(\Gamma;x+1,y+1)=\widetilde{\mathcal Q}(M;x,y,ax^2,by^2).
\]
Recall the following specializations of the Tutte polynomial:
$$T(\Gamma;x+1,1)=\sum_{\stackrel{A\subseteq E}{r(A)=|A|}}x^{r(\Gamma)-|A|},$$
$$T(\Gamma;1,y+1)=\sum_{\stackrel{A\subseteq E}{r(A)=r(E)}}y^{|A|-r(\Gamma)},$$
giving respectively generating functions for spanning forests of $\Gamma$ according to the number of edges and generating functions for connected spanning subgraphs.

\begin{proposition}\label{prop:quasi-trees_bouquets}
For a map $M = (V,E,F)$,
$$\widetilde{\mathcal Q}(M;x,0,1,1)=\sum_{\stackrel{A\subseteq E: \:\mbox{\rm \tiny conn. cpts of}}{\mbox{\rm \tiny $M\backslash A^c$ quasi-trees}}}x^{r(M/A)},$$
$$\widetilde{\mathcal Q}(M;0, y,1,1)=\sum_{\stackrel{A\subseteq E:\: \mbox{\rm \tiny conn. cpts of}}{\mbox{\rm \tiny $M/A$ bouquets}}}y^{r^*(M\backslash A^c)}.$$
\end{proposition}
\begin{proof} 
The first expression follows from the definition of $\widetilde{\mathcal Q}(M;x,y,a,b)$ (see~(\ref{eq:Q_renorm})) and the fact that $r^*(M\backslash A^c)=f(M\backslash A^c)-k(M\backslash A^c)=0$ if and only if each connected component of $M\backslash A^c$ has just one face (a quasi-tree). The second expression follows dually from the observation that $r(M/A)=v(M/A)-k(M/A)=0$ if and only if each connected component of $M/A$ has exactly one vertex (a bouquet). 
\end{proof}

\begin{corollary}\label{cor:Q_plane_1010}
For a map $M=(V,E,F)$,  
$$\widetilde{\mathcal Q}(M;1,0,1,1)=\#\{A\subseteq E:\mbox M\backslash A^c \:\mbox{\rm is a quasi-forest}\},$$ 
$$\widetilde{\mathcal Q}(M;0,1,1,1)=\#\{A\subseteq E: \mbox{\rm connected components of }M/A \:\mbox{\rm are bouquets}\}.$$
\end{corollary}

\noindent The evaluations of Corollary~\ref{cor:Q_plane_1010} are analogous (and for plane maps identical) to the following evaluations of the Tutte polynomial for a graph $\Gamma=(V,E)$, giving the number of spanning forests and number of connected spanning subgraphs: 
\begin{align*}T(\Gamma;2,1) & = \#\{A\subseteq E: n(\Gamma\backslash A^c)=0\}\\
& =\#\{A\subseteq E:\mbox{\rm connected components of } \Gamma\backslash A^c \:\mbox{\rm are trees}\},\end{align*}
and 
\begin{align*}T(\Gamma;1,2)& =  \#\{A\subseteq E: r(\Gamma\backslash A^c)=r(\Gamma)\}. \\
& =\#\{A\subseteq E:\mbox{\rm connected components of } \Gamma/A \:\mbox{\rm are single vertices with loops}\}.
\end{align*}

\section{Concluding remarks}

The surface Tutte polynomial of a map (orientable or non-orientable) contains as evaluations the number of nowhere-identity local $G$-flows and number of nowhere-identity local $G$-tensions, together with other evaluations analogous to those of the Tutte polynomial of a graph, such as the number of spanning quasi-trees of a connected map. Moreover, the surface Tutte polynomial coincides with the Tutte polynomial on plane maps and behaves with respect to surface (geometric) duality in the same way as the Tutte polynomial does with respect to matroid duality. This leads to the intriguing question as to what other properties of the Tutte polynomial (for planar graphs) lift up from plane maps to analogous properties of the surface Tutte polynomial of an arbitrary map. Here we highlight three areas that seem to us of significant interest. 

\paragraph{Evaluations that count}

Evaluations of the Tutte polynomial of a graph with combinatorial interpretations -- such as the number of acyclic orientations --  are also evaluations of the surface Tutte polynomial of an embedding of the graph as a map, as the surface Tutte polynomial of a map contains the Tutte polynomial of the underlying graph as a specialization. More interesting are evaluations of the surface Tutte polynomial with combinatorial-topological interpretations that do not depend on just the underlying graph of the map, but which nonetheless coincide with evaluations of the Tutte polynomial for plane graphs.
So, for example, is there an extension of the notion of acyclic and totally cyclic orientation to maps in the same vein as the extension of the notion of group-valued tensions and flows from graphs to maps as {\em local} tensions and flows, and spanning trees to spanning {\em quasi-}trees? And if so, are these objects enumerated by an evaluation of the surface Tutte polynomial? The chromatic polynomial of a graph evaluated at $-1$ gives the number of acyclic orientations: are there like interpretations of nowhere-zero local $\mathbb{Z}_{n}$-tensions and nowhere-zero local $\mathbb{Z}_{n}$-flows of a map for $n = -1$? The same question can be asked for other group sequences, such as the dihedral groups $(D_{2n})$, for which the number of local flows is a (quasi)polynomial in $n$.

\paragraph{Deletion--contraction recurrence and edge activities}

The Tutte polynomial is universal for graph invariants multiplicative over disjoint unions that satisfy a deletion-contraction recurrence applicable to all edge types (bridge, loop, ordinary); Tutte's universal $V$-function is universal for such graph invariants whose deletion-contraction recurrence is only applicable to non-loops (bridge, ordinary). The operation of edge contraction in a map does not usually correspond to contraction of the edge in the underlying graph; deletion and contraction are dual operations under surface duality (just as deletion and contraction of edges in graphs are dual at the level of cycle matroids of graphs). The Bollob\'as--Riordan polynomial of a map is universal for map invariants satisfying a deletion-contraction recurrence for non-loops with a value on (standard) bouquets that takes a special form (much as the Tutte polynomial of a graph can be thought of as a $V$-function taking a particularly simple form on graphs just consisting of loops). 
The Krushkal polynomial likewise satisfies a deletion-contraction recurrence for non-loops, so that its values on bouquets determine it. What form does a deletion-contraction recurrence take for the surface Tutte polynomial $\mathcal T(M;\x,\y)$,  or for its specialization $\widetilde{\mathcal{Q}}(M;x,y,a,b)$ defined in Definition~\ref{def:Q}? How many edge types does the recurrence involve? For the Tutte polynomial, and for $V$-functions more generally, there are three: ordinary, bridge, and loop. In our forthcoming paper~\cite{GLRV18+} on the specialization of the surface Tutte polynomial giving a Tutte polynomial for signed graphs (see Remark~\ref{rmk:signed}) we establish a deletion--contraction recurrence involving five edge types. Huggett and Moffatt~\cite{HM18} define edge activities for a map (and more generally for coloured ribbon graphs) and spanning quasi-tree activities analogously to the internally and externally active edges and spanning tree activities of a graph (only there are ten types of activity rather than two). This yields a similar expression for the coefficients of various specializations of the surface Tutte polynomial in terms of  spanning quasi-tree activities, including the Bollob\'as--Riordan polynomial and the Krushkal polynomial. 
Butler~\cite{butler12} gave a different quasi-tree expansion for the Krushkal polynomial (for orientable or non-orientable maps), although the terms involved are Tutte polynomials of graphs associated with submaps rather than the simpler terms featuring in~\cite{HM18}. Wang and Sachs~\cite{WS92} give a spanning tree expansion of Tutte's universal $V$-function according to internal and external acitivity. Can we adapt this and the approach of Huggett and Moffatt to obtain a quasi-tree expansion of the surface Tutte polynomial?

\paragraph{Knot invariants and signed graph invariants}
Thistlethwaite's construction of the Jones polynomial of an alternating link as the Tutte polynomial of a medial graph associated with the link prompts a similar search for new knot invariants obtained by specializing the surface Tutte polynomial of an associated map.

As sketched in Remark~\ref{rmk:signed}, the surface Tutte polynomial of a map has a trivariate specialization which is an invariant of the underlying signed graph of the map. This parallels the specialization of the surface Tutte polynomial of a map to the Tutte polynomial of its underlying graph. Furthermore, local $G$-flows of a signed graph, as defined by Bouchet~\cite{bouchet83}, are counted by an evaluation of this signed graph invariant, along with the dually defined local $G$-tensions of a signed graph. We therefore have an alternative candidate for a ``signed Tutte polynomial'' to that defined by Kauffman~\cite{kauf89}, one which has received mention in the slides of a talk given in 2013 by Krieger and O'Connor~\cite{kriocon13} but seems not to have survived this talk. This may be due to that fact that the signed graph invariant in question does not yield an interesting knot invariant via the usual medial graph construction in the same way that Kauffman's signed graph polynomial yields (by design) the bracket polynomial, nor does Thistlethwaite's result connecting the Jones polynomial of alternating knots and the
Tutte polynomial of a graph extend to this invariant of signed graphs for knots in general.
However, in its correspondence to the Tutte polynomial in its original conception as the dichromate of a graph, the new signed graph invariant surely merits further exploration.

\section*{Acknowledgement}
We thank Lex Schrijver for useful comments on an earlier version of this paper.

\appendix

\section{Proof of Lemma~\ref{lem:genus_submaps}}\label{s.proof_s}

\begin{proof}[Proof of Lemma~\ref{lem:genus_submaps}]
	 Using Euler's  formula~\eqref{eq:Euler} and $v(M\backslash A^c)=v(M)$, $f(M/A)=f(M)$,
	\begin{align}\label{align:total}
	s(M\backslash A^c)\!+\!s(M\!/\!A) & = 2k(M\backslash A^c)\!+\!2k(M\! /\! A)\!-\!v(M\backslash A^c)\!-\!v(M\! /\! A)\!+\!e(M\backslash A^c)\!+\!e(M\! /\! A) \nonumber\\
	& \quad\quad\quad\quad\quad-f(M\backslash A^c)\!-\!f(M\! /\! A)\nonumber\\
	& = s(M)+[k(M\backslash A^c)\!\!-k(M)\!-\!f(M\backslash A^c)\!\!+k(M\! /\! A)]\!
	\nonumber\\
	& \quad\quad\quad\quad\quad +\![k(M\! /\! A)\!-k(M)\!-\!v(M\! /\! A)+\!k(M\backslash A^c)].
	\end{align}
	We now claim that for all $A \subseteq E$,
	\begin{equation}
	k(M\! /\! A)\!-k(M)\!-\!v(M\! /\! A)+\!k(M\backslash A^c)\leq 0.\label{eq:components}
	\end{equation}
	If (\ref{eq:components}) is true, then 
	\begin{align}\label{align:duaal}
	k(M\backslash A) - k(M) - f(M\backslash A) + k(M/A^c) = \\
	k((M\backslash A)^*) - k(M^*) - v((M\backslash A)^*) + k((M/A^c)^*) = \nonumber\\
	k(M^*/A^*)-k(M^*) - v(M^*/A^*) + k(M^*\backslash (A^*)^c) \leq 0,\nonumber
	\end{align}
	where in the first equality we move to the dual and in the second equality we use Proposition~\ref{p.del_dual_cont}. Complementing $A$ in (\ref{align:duaal}) yields
	\begin{equation}\label{equation:twee}
	k(M\backslash A^c) - k(M) - f(M\backslash A^c) + k(M/A) \leq 0,
	\end{equation}
	so that (\ref{eq:components}) and (\ref{equation:twee}) together with (\ref{align:total}) prove the lemma. \\
	\indent It remains to prove the inequality~\eqref{eq:components}, which it suffices to show for a connected map $M$ as $v(\cdot)$ and $k(\cdot)$ are additive over disjoint unions. Assume then that $M$ is connected. We prove that there exists a partition $A=B_1\cup \cdots \cup B_t$ of the edges in $A$ into non-empty sets with the following property: if we define $A_i := \cup_{j=1}^{i}B_j$, for $1 \leq i \leq t$, and $A_0=\emptyset$, then 
	\begin{equation}\label{eq.cont1}
	k(M\! /\! A_{i})\!-\!v(M\! /\! A_{i})+\!k(M\backslash A_{i}^c)\leq
	k(M\! /\! A_{i-1})\!-\!v(M\! /\! A_{i-1})+\!k(M\backslash A_{i-1}^c)\; ,\end{equation}
	for $1 \leq i \leq t$. The right-hand side of inequality~\eqref{eq.cont1} for $i=1$ is 
	\[k(M\! /\! A_{0})-v(M\! /\! A_{0})+k(M\backslash A_{0}^c)=
	k(M)-v(M)+k(M\backslash E)=k(M),\]
	and so the existence of such a partition proves inequality~\eqref{eq:components}.
	
As a basis for induction we first prove inequality~\eqref{eq.cont1} for $i=1$. 
	Let $B_1:=F\subseteq A$ be a maximal spanning forest of $M\backslash A^c$ (that is, $M\backslash F^c$ contains no cycles and for each $e\in A\backslash F$ the underlying graph of  $M\backslash (F\cup\{e\})^c$ contains a cycle of $M\backslash A^c$). 
	We then have $k(M\backslash F^c)=k(M\backslash A^c)$. 
		Each edge in $F$ is a non-loop of $M$. Contracting a non-loop in $M$ corresponds to its contraction in the underlying graph $\Gamma$, and in particular preserves connectivity of $M$ (the edge is deleted, its endpoints $u$ and $v$ are fused into a single vertex, whose incident edges are the other edges incident with $u$ or $v$).
	We thus have $k(M/F)\!-\!k(M)=1-\!1=0= \!v(M/F)-\!k(M\backslash F^c)$, as each connected component of $M\backslash F^c$ is contracted to a map with a single vertex in~$M/F$. Hence
	\[
	k(M\! /\! F)\!-\!v(M\! /\! F)+\!k(M\backslash F^c)=1=
	k(M\! /\!\; \emptyset)\!-\!v(M\! /\!\; \emptyset)+\!k(M\backslash \emptyset^c),
	\]
	where the latter equality is due to $M$ being a connected map. Thus inequality~\eqref{eq.cont1} holds with equality for $i=1$.

Suppose inductively that we have constructed non-empty sets $B_1,\dots, B_{i}$ and $A_{i}=B_{i}\cup A_{i-1}$ for some $i\geq 1$ such that inequality~\eqref{eq.cont1} holds. We have already verified the base case $i=1$ with $B_1=F=A_1$, and $A_0=\emptyset$. 	

Inequality~\eqref{eq.cont1} with $i+1$ in place of $i$ simplifies as follows. Let $B$ be a non-empty subset of $A\setminus F$ (if $A = F$ then $t=1$ and nothing remains to be proved). Since $k(M\backslash F^c)=k(M\backslash A^c)$ and 
	$k(M\backslash F^c)\geq k(M\backslash (F\cup B)^c)\geq k(M\backslash A^c)$, it follows that 
	$k(M\backslash F^c)= k(M\backslash (F\cup B)^c)$. Therefore $k(M\backslash A_{i}^c) = k(M\backslash A_{i-1}^c)$ for $i \geq 2$. Thus, for $i \geq 1$, inequality~\eqref{eq.cont1} with $i+1$ in place of $i$ now reads as
	\begin{equation}\label{eq.cont}
	k(M\! /\! A_{i+1})\!-\!v(M\! /\! A_{i+1})\leq
	k(M\! /\! A_{i})\!-\!v(M\! /\! A_{i}).\end{equation}
	We show how to define $B_{i+1}$ so that inequality~\eqref{eq.cont} holds with $A_{i+1}=B_{i+1}\cup A_i$. 

\indent Pick an edge $e\in A\backslash A_i$, which in $M/F$ is a loop, as there is a unique cycle of $M\backslash A^c$ whose edges are contained in $F\cup\{e\}$. Inductively we assume $M/A_i$ also consists of just loops, for we shall see that under all possible choices of set $B_{i+1}$ the map $M/A_{i+1}=(M/A_i)/B_{i+1}$ consists of just loops. (As a base, when $i=1$ we have $A_1=F$ and all edges of $M/A_1$ are loops.)

Let $M/A_i=(\theta,\sigma,\tau;C)$ and $e=\{a,\theta a,\sigma a,\theta\sigma a\}$.

\paragraph{Case $\text{I}$ (twisted):}
Suppose that $e$ is a twisted loop of $M/A_i$. 
	Then there is a pair of cycles of $\tau$ of the form 
$$\left(\begin{array}{cccc} a & X & \theta a & Y\end{array}\right)\:\left(\begin{array}{cccc} \sigma a & Y^{-1} & \theta\sigma a & X^{-1}\end{array}\right),$$
for some (possibly empty) sequences of crosses $X$ and $Y$.
\noindent By Observation~\ref{obs:contraction}, 
the permutation $\tau''$ in $M/(A_i\cup \{e\}) = (\theta'',\sigma'',\tau'';C'')$ has the same cycles as $\tau$,
	except for the above pair of cycles containing crosses in $e$, which are replaced by 
$$\left(\begin{array}{cc}  X & Y^{-1}\end{array}\right)\:\left(\begin{array}{cc} X^{-1} & Y\end{array}\right).$$
%
	%
	
	\noindent In particular we have that
	\begin{itemize}
		\item $v(M/ (A_i\cup \{e\}))=v(M/A_i)$, as the number of cycles in $\tau$ and $\tau''$ is the same (some cycles in the permutation may be empty),
		
		\item $k(M/(A_i\cup \{e\}))=k(M/A_i)$, as the connectivity between vertices in $M/A_i$ is preserved upon contracting $e$ (the loop $e$ is irrelevant for connectivity between vertices in $M/A_i$).
	\end{itemize}

	\noindent Thus, if $e$ is a twisted loop of $M/A_i$, setting $B_{i+1} := \{e\}$, $A_{i+1}=A_i\cup\{e\}$, gives equality in~\eqref{eq.cont}.

\paragraph{Case $\text{II}$ (non-twisted):}
	If $e$ is a non-twisted loop in $M/A_i$, then, keeping the same notation, the cycles of $\tau$ containing crosses of $e$ are of the form
$$\left(\begin{array}{cccc} a & X & \theta\sigma a & Y\end{array}\right)\:\left(\begin{array}{cccc} \sigma a & Y^{-1} & \theta a & X^{-1}\end{array}\right).$$

	%
	%
\noindent  By Observation~\ref{obs:contraction}, 
the permutation $\tau''$ in $M/(A_i\cup \{e\}) = (\theta'',\sigma'',\tau'';C'')$ has the same cycles as $\tau$,
	except for the above pair of cycles containing crosses in $e$, which are replaced by  the two pairs of cycles 
$$\left(\begin{array}{c}  X \end{array}\right)\:\left(\begin{array}{c} X^{-1}\end{array}\right)\quad\mbox{ and }\quad \left(\begin{array}{c}  Y\end{array}\right)\:\left(\begin{array}{c} Y^{-1} \end{array}\right),$$
corresponding to the two vertices $u$ and $v$ into which the vertex incident with $e$ is split upon contracting $e$. 
Now there are three possibilities:

	\begin{description}
		\item[Case $\text{II.i}$:] The vertices $u$ and $v$ are not connected in $M/(A_i\cup \{e\})$.
		Then $v(M/(A_i\cup \{e\}))=v(M/A_i)+1$ and $k(M/(A_i\cup \{e\}))=k(M/A_i)+1$, so that
			\[
		k(M/(A_i\cup \{e\}))-v(M/(A_i\cup \{e\}))=k(M/A_i)-v(M/A_i).
		\]
	\item[Case $\text{II.ii}$:]  There is an edge $e'\in A\setminus (A_i\cup \{e\})$ with endpoints $u$ and $v$. Then
		$v(M/(A_i\cup \{e,e'\}))=v(M/(A_i\cup \{e\}))-1=v(M/A_i)$ while 
		$k(M/(A_i\cup \{e,e'\}))=k(M/(A_i\cup \{e\}))=k(M/A_i)$. Hence
		\[
		k(M/(A_i\cup \{e,e'\}))-v(M/(A_i\cup \{e,e'\}))=k(M/A_i)-v(M/A_i).
		\]
%
%
%
%
%
%
	\item[Case $\text{II.iii}$:] The vertices $u$ and $v$ are connected in $M/(A_i\cup\{e\})$ but they are not the endpoints of an edge in $A\setminus (A_i\cup \{e\})$. Then $v(M/(A_i\cup \{e\}))=v(M/A_i)+1$ while $k(M/A_i)=k(M/(A_i\cup\{e\}))$. Hence
	\[
	k(M/(A_i\cup \{e\}))-v(M/(A_i\cup \{e\}))<k(M/A_i)-v(M/A_i).
	\]
\end{description}
Therefore, if we define
\[
B_{i+1} := \begin{cases} \{e\} \hspace{6.5mm} \text{ in Case $\text{I}$, Case $\text{II.i}$ and Case $\text{II.iii}$,}\\  \{e,e'\} \hspace{2mm} \text{ in Case $\text{II.ii}$},\end{cases}
\]
then inequality~\eqref{eq.cont} is satisfied. 

Finally, we verify that all the edges of $M/A_{i+1}$ are loops. 
All loops in $M/A_i$ remain loops in $M/A_{i+1}$ that are not incident with the same vertex as the loop $e\in A\backslash A_i$ picked at the outset. 
A loop on the same vertex as $e$ remains a loop in  $M/A_{i+1}$ in Case $\text{I}$, Case $\text{II.i}$ and Case $\text{II.iii}$ as here $B_{i+1}=\{e\}$ and contraction of $e$ to obtain $M/A_{i+1}$ from $M/A_i$ does not create any non-loop edges. 
In Case II.ii the non-loop edge $e'$ produced upon contracting $e$ is added to $B_{i+1}$, so that $M/A_{i+1}$ again consists of just loops. 


By induction the inequality~\eqref{eq.cont1}  is satisfied for all $1\leq i\leq t$, where $A_t=A$.
\end{proof} 

\section{Enumerating flows on standard bouquets}\label{sec:proof of flow count}

While several proofs of Theorem~\ref{thm:flows standard bouquets} have appeared~\cite{frob96, jones95, mednykh78}, the proof that follows is new and elementary, requiring only basic representation theory. We begin by collecting together from the first twenty pages of Serre's book~\cite{serre77} the relevant facts needed for our proof. 

Let $G$ be a finite group and let $\chi_{\text{reg}}$ denote the character of the \textit{regular representation} of $G$. This satisfies the identity
\begin{equation}\label{equation:reg1}
\chi_{\text{reg}}(g) = \begin{cases} |G| \hspace{4mm}\text{ if } g=1, \\ 0 \hspace{7mm}\text{ otherwise}.\end{cases}
\end{equation}
Let $\widehat{G}$ denote the set of all irreducible representations of $G$ up to equivalence. Then
\begin{equation}\label{equation:reg2}
\chi_{\text{reg}} = \sum_{\rho \in \widehat{G}}n_{\rho}\chi_{\rho},
\end{equation}
where $n_{\rho}$ is the dimension of the irreducible representation $\rho$ with character $\chi_{\rho}$. 

Without loss of generality we may assume that each  $\rho\in\widehat{G}$ is unitary, that is,  $\rho(g)$ is a unitary matrix for each $g \in G$. Let $\rho, \rho' \in \widehat{G}$. Then for $1 \leq i,j \leq n_{\rho}$ and $1 \leq i',j' \leq n_{\rho'}$ we have the orthogonality relation
\begin{equation}\label{equation:orthorel}
\sum_{g \in G}\rho(g)_{i,j}\rho'(g^{-1})_{j',i'} = \begin{cases} \frac{|G|}{n_{\rho}} \hspace{3mm}\text{ if } \rho = \rho', i = i' \text{ and } j=j', \\ 0 \hspace{6mm}\text{ otherwise}.\end{cases}
\end{equation}

The \textit{contragredient representation} (or dual representation) $\rho^*$ of an irreducible representation $\rho: G \rightarrow \text{GL}(V)$ is defined on the dual space $V^*$ of $V$ by 
\begin{equation}\label{equation:duale}
\rho^*(g) := \rho(g^{-1})^T.
\end{equation} 
If $V$ is self-dual via a $G$-invariant symmetric bilinear form 
then 
this form can be chosen so that $\rho^*(g) = \rho(g)$ for all $g \in G$. If $V$ is self-dual via a $G$-invariant skew-symmetric bilinear form then this form can (and will be) chosen so that 
\begin{equation}\label{equation:dual}
\rho^*(g)_{i,j} = (-1)^{\alpha(i,j)}\rho(g)_{i+n_{\rho}/2,j+n_{\rho}/2} \hspace{4mm} \forall i,j \in [n_{\rho}],
\end{equation}
where the indices are taken modulo $n_{\rho}$ and where
\begin{equation}\label{equation:alpha}
\alpha(i,j) = \begin{cases} 1 \hspace{5mm}\text{ if } |i-j| \geq n_{\rho}/2, \\ 0 \hspace{6mm} \text{otherwise}.\end{cases}
\end{equation}

\noindent Note that if $V$ is isomorphic to $V^*$ via a $G$-invariant skew-symmetric bilinear form then the dimension $n_{\rho}$ of $V$ must be even.

Recall from equation~\eqref{fsindicator} that for
a representation $\rho$ with character $\chi_{\rho}$ the Frobenius-Schur indicator $\mathcal{F}(\rho)$ is defined by
$\mathcal{F}(\rho) = \frac{1}{|G|}\sum_{g \in G}\chi_{\rho}(g^2)$.
\begin{theorem}[Frobenius, Schur \cite{frob06}]\label{theorem:fs}
	If $\rho: G \rightarrow \mathrm{GL}(V)$ is an irreducible representation then
	\[
	\mathcal{F}(\rho) = \begin{cases} -1 \hspace{2mm}\text{ if and only if $V$ has a non-zero $G$-invariant skew-symmetric bilinear form,} \\ 0 \hspace{5mm}\text{ if and only if $V$ has no non-zero $G$-invariant bilinear form}, \\ 1 \hspace{5mm}\text{ if and only if $V$ has a non-zero $G$-invariant symmetric bilinear form.}\end{cases}
	\]	
\end{theorem}

For a more modern account of the above result, see Theorem $23.16$ in \cite{james01}.

\begin{lemma}\label{lemma:fsrep}
Let $\rho \in \widehat{G}$. Then
\begin{align*}
\sum_{g \in G}\rho(g)\rho(g) = \frac{\FF(\rho)|G|}{n_{\rho}}I_{n_{\rho}},
\end{align*}
where $I_{n_{\rho}}$ is the identity matrix of size $n_{\rho}$.
\end{lemma}
\begin{proof}
By equation~\eqref{equation:duale} the $(i,j)$-entry of the following sum of $n_{\rho}\times n_{\rho}$-matrices is calculated to be
\begin{equation}\label{equation:ij}
\Big(\sum_{g \in G}\rho(g)\rho(g)\Big)_{i,j} = \sum_{t = 1}^{n_{\rho}}\sum_{g}(\rho^*(g^{-1})^T)_{i,t}\rho(g)_{t,j}.\end{equation}
We calculate the expression on the right-hand side of equation (\ref{equation:ij}) by considering the possible values of $\mathcal{F}(\rho)$. 

If $\mathcal{F}(\rho)=0$ then by Theorem \ref{theorem:fs} $V$ and $V^*$ are not isomorphic. Hence $\rho$ and $\rho^*$ are not equivalent and the orthogonality relation~\eqref{equation:orthorel} implies that equation~\eqref{equation:ij} is equal to zero in this case. \\
\indent If $\mathcal{F}(\rho) = 1$ then the same orthogonality relation yields
\[
\sum_{t}\sum_{g}\rho(g^{-1})_{t,i}\rho(g)_{t,j} = \begin{cases} \frac{|G|}{n_{\rho}} \hspace{4mm} \text{ if } i = j,\\ 0 \hspace{8.5mm} \text{otherwise,}\end{cases}
\]
proving the assertion for such $\rho$.\\
\indent Assume now that $\mathcal{F}(\rho)=-1$. Then we calculate that
\begin{align}\label{equation:skewsymm}
\sum_{t}\sum_{g}(\rho^*(g^{-1})^T)_{i,t}\rho(g)_{t,j} &= \sum_{t}\sum_{g}(-1)^{\alpha(i,t)}\rho(g^{-1})_{t+n_{\rho}/2,i+n_{\rho}/2}\rho(g)_{t,j}\\
&= \begin{cases} \frac{-|G|}{n_{\rho}} \hspace{4mm} \text{ if } i = j,\\ 0 \hspace{11mm} \text{otherwise.}\end{cases}\nonumber
\end{align}
Equation (\ref{equation:dual}) is used in the first equality. The orthogonality relation (\ref{equation:orthorel}) forces the index $t$ in the summation in (\ref{equation:skewsymm}) to equal both $j+n_{\rho}/2$ and $i+n_{\rho}/2$ in order to obtain a non-zero contribution to the sum. In that case $i = j$ and then $\alpha(i,t) = \alpha(i,i+n_{\rho}/2) = 1$ (for all $i$) by equation~\eqref{equation:alpha}.
\end{proof}

We now have all the results we need in order to prove Theorem~\ref{thm:flows standard bouquets}. 
\begin{proof}[Proof of Theorem~\ref{thm:flows standard bouquets}]
We first assume that $M$ is non-orientable. 
Writing $M=(\theta,\sigma,\tau;C)$, with $\tau$ as in \eqref{eq:non-orientable sb}, we see that we obtain the single tension equation 
\[
h_1^2\cdots h_g^2h_{g+1}h_{g+1}^{-1}\cdots h_{n^*}h_{n^*}^{-1}=1
\]
with $h_i\in G$, for $i=1,\ldots,n^*$.
Note that we can choose any group element for $h_{i}$ with $g+1 \leq i \leq n^*$ as they do not contribute to the tension condition. This accounts for a factor $|G|^{n^*-g}$. 
We then write
\[
p_G^1(M) = |G|^{n^*-g}\sum_{h_1,\dots,h_g\in G}\bold{1}\left(\prod_{i=1}^{g}h_i^2\right),
\]
where $\bold{1}(h)$ is one if and only if $h = 1$ and zero otherwise. 
By equations (\ref{equation:reg1}) and (\ref{equation:reg2}) we have 
\begin{align}\label{equation:compu}
p_G^1(M) &\stackrel{\eqref{equation:reg1}}{=} |G|^{n^*-g-1}\sum_{h_1,\dots,h_g \in G}\chi_{\text{reg}}(h_1^2\cdots h_g^2)\nonumber\\
&\stackrel{\eqref{equation:reg2}}{=} |G|^{n^*-g-1}\sum_{h_1,\dots,h_g \in G}\sum_{\rho \in \widehat{G}}n_{\rho}\chi_{\rho}(h_1^2\cdots h_g^2).
\end{align}
With the aid of Lemma \ref{lemma:fsrep} (in the fourth equality below) we see that
\begin{align*}
p_G^1(M) &=|G|^{n^*-g-1}\sum_{\rho \in \widehat{G}}n_{\rho}\sum_{h_1,\dots,h_g \in G}\chi_{\rho}(h_1^2\cdots h_g^2) \\
&= |G|^{n^*-g-1}\sum_{\rho}n_{\rho}\sum_{h_1,\dots,h_g}\text{tr}(\rho(h_1)\rho(h_1)\cdots\rho(h_g)\rho(h_g))\\
&= |G|^{n^*-g-1}\sum_{\rho}n_{\rho}\text{tr}\Big(\big(\sum_{h_1}\rho(h_1)\rho(h_1)\big)\cdots\big(\sum_{h_g}\rho(h_g)\rho(h_g)\big)\Big)\\
&=|G|^{n^*-g-1}\sum_{\rho}n_{\rho}\text{tr}\Big(\big(\frac{\mathcal{F}(\rho)|G|}{n_{\rho}}I_{n_{\rho}}\big)^{g}\Big)\\
&= |G|^{n^*-1}\sum_{\rho \in \widehat{G}}\FF(\rho)\matchsupheight^gn_{\rho}^{2-g},
\end{align*}
which finishes the proof for non-orientable standard bouquets.

In case $M$ is orientable, a similar but easier computation gives Theorem~\ref{thm:flows standard bouquets} for that case.
(See~\cite{goodall16,litjens17} for details.)
\end{proof}

 \bibliographystyle{abbrv} 
 \bibliography{surfaceTutte}

\end{document}